\documentclass[11pt]{article}
\usepackage{amsfonts,mathrsfs}
\usepackage{amssymb}
\usepackage{amsmath}
\usepackage{graphicx}
\usepackage{hyperref}
\usepackage{color}
=0pt
\def\sqr#1#2{{\vcenter{\vbox{\hrule height.#2pt
              \hbox{\vrule width.#2pt height#1pt \kern#1pt \vrule
width.#2pt}
              \hrule height.#2pt}}}}
\def\signed #1{{\unskip\nobreak\hfil\penalty50
              \hskip2em\hbox{}\nobreak\hfil#1
              \parfillskip=0pt \finalhyphendemerits=0 \par}}
\def\endpf{\signed {$\sqr69$}}
\def\dbE{{\mathbb{E}}}
\def\dbF{{\mathbb{F}}}

\def\dbR{{\mathbb{R}}}

\def\e{\varepsilon}

\def\t{\tau}
\def\f{\varphi}

\def\3n{\negthinspace \negthinspace \negthinspace }
\def\2n{\negthinspace \negthinspace }
\def\1n{\negthinspace }
\def\ns{\noalign{\smallskip} }

\def\ds{\displaystyle}
\def\D{\Delta}

\def\Om{\Omega}

\def\cC{{\cal C}}

\def\cF{{\cal F}}

\def\cJ{{\cal J}}

\def\cL{{\cal L}}

\def\cP{{\cal P}}

\def\cU{{\cal U}}

\def\cW{{\cal W}}
\def\cX{{\cal X}}

\def\mE{{\mathbb{E}}}

\def\no{\noindent}

\def\ss{\smallskip}

\def\bs{\bigskip}
\def\q{\quad}
\def\qq{\qquad}
\def\hb{\hbox}

\def\lan{\big\langle}
\def\ran{\big\rangle}

\def\pa{\partial}

\def\wt{\widetilde}
\def\cd{\cdot}

\def\ae{\hbox{\rm a.e.{ }}}

\def\deq{\mathop{\buildrel\D\over=}}

\def\({\Big (}
\def\){\Big )}
\def\[{\Big[}
\def\]{\Big]}

\def\={\buildrel \triangle \over =}

\newtheorem{lemma}{Lemma}[section]
\newtheorem{remark}{Remark}[section]

\newtheorem{theorem}{Theorem}[section]
\newtheorem{corollary}{Corollary}[section]

\newtheorem{definition}{Definition}[section]
\newtheorem{proposition}{Proposition}[section]

\allowdisplaybreaks
 \makeatletter
   
   \@addtoreset{equation}{section}
\makeatother

\title{Transposition Approach to Optimal Control of McKean-Vlasov SPDEs}

\author{Liangying Chen\footnote{Technische Universit\"{a}t Berlin and Freie Universit\"{a}t Berlin, Berlin, Germany, Email: liangying.chen@fu-berlin.de}
~~~ and ~~~ Wilhelm Stannat\footnote{Technische Universit\"{a}t Berlin, Berlin, Germany, Email: stannat@math.tu-berlin.de}}
\date{}

\begin{document}

\maketitle

\begin{abstract}

In this paper, we investigate an optimal control problem for McKean–Vlasov stochastic partial differential equations, in which the coefficients depend on the law of the state process. For systems with nonconvex control sets, we establish a Pontryagin-type stochastic maximum principle that provides necessary optimality conditions for admissible controls. The analysis is based on the classical spike variation method together with the introduction of an adjoint backward stochastic partial differential equation involving Lions derivatives with respect to probability measures. Our results extend the stochastic maximum principle for McKean-Vlasov controlled stochastic differential equations to the infinite-dimensional SPDE setting. 

\end{abstract}

\bs

\no{\bf 2010 Mathematics Subject
Classification}. 93E20, 60H15

\bs

\no{\bf Key Words:} Pontryagin type maximum principle, Stochastic optimal control, McKean-Vlasov equation, Transposition solution, Stochastic distributed
parameter systems

\section{Introduction}

We consider the optimal control of the following semilinear McKean–Vlasov stochastic evolution equation 
\begin{equation}
\label{system0}
\begin{cases}
dX(t) = AX(t)dt+a(t,X(t),\mathcal L(X(t)), u(t))dt+b(t,X(t),\mathcal L(X(t)),u(t))dW(t),\\ \hfill t\in (0,T] \\
X(0) = \xi \in L^2_{\mathcal F_0}(\Omega;H).
\end{cases}    
\end{equation}
on a separable real Hilbert space $H$ and with cost functional  
\begin{equation}
\label{cost1}
\mathcal{J}(u(\cdot)) 
= \mathbb{E}\Big(\int_0^T f(t,X(t),\mathcal L(X(t)),u(t))dt+h(X(T),\mathcal L(X(T)))\Big).
\end{equation}
In \eqref{system0} $A$ is the generator of a $C_0$-semigroup $\{S(t)\}_{t\ge 0}$ on $H$, 
$\{W(t)\}_{t\ge 0}$ is a cylindrical Wiener process, defined an a complete filtered 
probability space $(\Omega , \mathcal{F}, \mathbb F\triangleq \{\cF_t\}_{t\geq 0} ,
\mathbb{P})$, with $\mathbb{F}$ satisfying the usual conditions. 
$\{W(t)\}_{t\ge 0}$ takes values in a possibly different separable real Hilbert space 
$\tilde{H}$, $a$ and $b$ denote suitable coefficients, whose precise assumptions will be 
specified below, and $\mathcal{L}(X(t))$ denotes the law of $X(t)$. The controls $u(t)$ 
take values in a (not necessarily convex) subset $U$ of a separable metric space and they 
are adapted to $\mathbb{F}$.

Our goal will be to establish a Pontryagin's type maximum principle for the optimal control 
problem  
\textbf{Problem} $\boldsymbol{(S)}$.  For any given $\xi\in L_{\mathcal F_0}^2 
(\Omega ; H)$, find a $\bar{u}(\cdot)\in \mathcal{U}[0,T]$ such that \vspace{-1mm}
\begin{equation}\label{OP2}
    \mathcal{J}(\bar{u}(\cdot))=\inf\limits_{u(\cdot)\in \mathcal{U}[0,T]}\mathcal{J}(u(\cdot)).\vspace{-1mm}
\end{equation}
where the set of admissible controls is defined as 
\begin{equation*}
\mathcal{U}[0,T]\triangleq\big\{u:[0,T]\times\Omega\to U\big| u \mbox{ is
$\mathbb{F}$-adapted}\big\}.
\end{equation*}
Any $\bar{u}(\cdot)\in \mathcal{U}[0,T]$ satisfying \eqref{OP2} is called an {\it optimal control} (of \textbf{Problem} $\boldsymbol{(S)}$). The corresponding state
$\overline{X}(\cdot)$ is called an {\it optimal state}, and
$(\overline{X}(\cdot),\bar{u}(\cdot))$ is called an {\it optimal pair}.

The study of optimal control problems for McKean–Vlasov dynamics is motivated by the analysis of the asymptotic behavior of large-population stochastic systems with mean-field interactions, such as those arising in models of financial markets, firms, and interacting agents \cite{Carmona2018}. In such settings, the state dynamics of each individual depend not only on its own state and control but also on the distribution of the entire population. Due to their broad range of applications, they have attracted considerable attention in recent years.

Several approaches have been developed to address the optimal control of McKean–Vlasov dynamics. The approach of particular interest in this paper is the Pontryagin maximum principle, which was originally established in the early 1950s by Pontryagin and his collaborators \cite{Pontryagin1962} for deterministic optimal control problems.

The literature on optimal control problems for McKean–Vlasov stochastic differential equations is relatively well developed. In particular, Pontryagin-type maximum principles for McKean–Vlasov control problems with convex control sets have been established in several works; see, e.g., \cite{Andersson2011, Carmona2015, Carmona2018, Hocquet2021}. Subsequently, by extending Peng’s approach \cite{Peng1990}, these results were generalized to the case of nonconvex control domains in \cite{Buckdahn2011, Buckdahn2016, Buckdahn2017, Spille2026}.  

In contrast, the literature on infinite-dimensional McKean–Vlasov control problems remains rather limited. Existing works mainly focus on specific settings. For instance, \cite{Ahmed2016} studies such problems under relaxed controls, while \cite{Dumitrescu2018} considers the case of partial information. Results for convex control domains have been obtained in \cite{Spille2025, Tang2019}. However, to the best of our knowledge, there is no existing work addressing optimal control problems for McKean–Vlasov stochastic partial differential equations with general nonconvex control domains. 

The generalization to infinite-dimensional nonconvex control setting faces several obstacles. First, the proof of Pontryagin's maximum principle in the finite dimensional case given in [5], based on the introduction of a stochastic exponential process and the application of the Martingale representation theorem to discuss the order of certain remainders in the Taylor expansion of the controlled state equations and the cost functional, cannot be carried over to the infinite dimensional case. 
This is the case because in the infinite dimensional setting the corresponding stochastic exponential process will not exist because of the unboundedness of the 
linear operator $A$, and the Martingale Representation Theorem would have to be applied to an operator-valued stochastic process, for which there's no suitable formulation available.
Nevertheless, inspired by \cite{Guatteri2025}, where delay equations are studied, we adopt an alternative approach based on introducing a backward stochastic evolution equation (BSEE). This allows us to derive the desired estimate without relying on stochastic exponentials. We refer to Proposition \ref{prop4.5-1} for a detailed discussion.

Second, as already observed in the literature on optimal control of stochastic partial differential equations (SPDEs); see, e.g., \cite{Chen2023, Wessels2025, Fabbri2017, Fuhrman2013, Lu2014, Lu2021, Stannat2021, Stannat2024}, the second-order adjoint state appearing in Pontryagin-type maximum principles is a BSEE taking values in the space 
$\mathcal L(H)$ of bounded linear operators on $H$. This issue has been addressed in several recent works, see \cite{Lu2014, Lu2015, Lu2021} and \cite{Stannat2021}.

To treat the first- and second-order adjoint equations, we adopt the transposition approach introduced in \cite{Lu2021}. In the McKean–Vlasov SPDE setting, the first-order adjoint equation takes the form of a McKean–Vlasov BSEE, whose well-posedness is established in Section 4.
The second-order adjoint equation, by contrast, remains a standard operator-valued BSEE. Its well-posedness has already been proved in \cite{Lu2021}. Nevertheless, due to the technical complexity and relatively recent development of this concept, we include its definition in our paper for completeness and ease of reference.

The third major obstacle to overcome is the definition of the Lions derivatives of the 
coefficients $a$, $b$, $f$ and $h$ of the state equation and the cost functional needed 
in the formulation of Pontryagin's maximum principle. To this end we will use the recent 
progress in \cite{Stannat2025} on the Lions derivative for mappings taking values in 
infinite dimensional state spaces. 

As it is necessary in the stochastic case to consider a second order Taylor expansion of 
the cost functional, due to the unbounded variation of the driving Wiener process, this 
also involves second order derivatives of the coefficients involving $L$-derivatives up to 
second order. It was a crucial observation in \cite{Buckdahn2016}
that in the formulation of Pontraygin's maximum principle for optimal controls of 
McKean-Vlasov stochastic differential equations, only second order derivatives of the 
coefficients of the type $\partial_{xx}\phi$ and $\partial_{y\mu} \phi$ appear.  
Derivatives of the type $\partial_{x\mu}\phi$ and $\partial_{\mu\mu}\phi$ can be 
neglected since they involve certain (conditional) expectations smoothing the unbounded 
variation in the underlying controlled stochastic evolution process. This observation can be also verified in the present infinite dimensional setting. The corresponding statement 
is formulated in Proposition \ref{prop4.5-1} and Corollary \ref{cor3.1} below.

 In summary, this is the first work to study optimal control problems for infinite dimensional McKean-Vlasov SPDEs with nonconvex control domains. Our approach is based on the transposition method combined with Lions derivatives for coefficients taking values in infinite dimensions. This allows us to work under a general filtration, rather than restricting to the natural filtration, thereby a broader class of stochastic systems. Within this framework, we establish a Pontryagin-type maximum principle for McKean–Vlasov SPDEs with nonconvex control domains, where the control enters the diffusion term and the filtration is general, a setting which has not been previously treated in the infinite-dimensional case.

 The rest of the paper is organized as follows: In Section 2 we formulate precise assumptions on \eqref{system0} and \eqref{cost1}. We state the first and second-order adjoint states and formulate our main result on Pontryagin maximum principle in Theorem \ref{thm4.1}. Section 3 states and further analyzes first and second order variational equations needed for the Taylor expansion of the cost functional. 
Section 4 deals with well-posedness of the adjoint states within the theory of (relaxed) transposition solutions. Section 5 provides the proof of Theorem \ref{thm4.1}. The final Section 6 provides an illustrative example to our McKean-Vlasov SPDE control problem.



\section{Mathematical setting and main result}
Let us first introduce some notations that will be used in our paper.

Let $\cX$ be a Banach space. For any $t\in[0,T]$ and $p\in
[1,\infty)$, denote by $L_{\cF_t}^p(\Om;\cX)$ the Banach space of
all $\cF_t$-measurable random variables $\xi:\Om\to \cX$ such that
$\mathbb{E}|\xi|_\cX^p < \infty$, with the canonical  norm. Denote
by $L^{p}_{\dbF}(\Om;C([t,T];\cX))$ the Banach space of all
$\cX$-valued $\mathbb{F}$-adapted continuous processes
$\phi(\cdot)$, with the norm\vspace{-3mm}
$$
|\phi(\cd)|_{L^{p}_{\dbF}(\Om;C([t,T];\cX))} \=
\Big[\mE\sup_{\tau\in
[t,T]}|\phi(\tau)|_\cX^p\Big]^{1/p}.  $$
Also, denote by $C_{\dbF}([t,T];L^{p}(\Om;\cX))$ the Banach space of
all $\cX$-valued $\mathbb{F}$-adapted processes $\phi(\cdot)$ such
that $\phi(\cdot):[t,T] \to L^{p}_{\cF_T}(\Om;\cX)$ is continuous,
with the norm
$$
|\phi(\cd)|_{C_{\dbF}([t,T];L^{p}(\Om;\cX))} \= \sup_{\tau\in
[t,T]}\left[\mE|\phi(\tau)|_\cX^p\right]^{1/p}.
$$
Write $D_{\dbF}([0,T];L^{r}(\Om;\cX))$ for the vector space of all
$\cX$-valued, $r$th power integrable $\dbF$-adapted processes
$\phi(\cdot)$ such that $\phi(\cdot):[0,T] \to
L^{r}_{\cF_T}(\Omega;\cX)$ is c\`adl\`ag, i.e., right
continuous with left limits. Clearly,
$D_{\dbF}([0,T];L^{r}(\Omega;\cX))$ is a Banach space with the norm:
$$
|\phi(\cd)|_{D_{\dbF}([0,T];L^{r}(\Omega;\cX))} = \sup_{t\in
[0,T)}\left(\mE|\phi(r)|_\cX^r\right)^{1/r}.
$$
Fix any $p_1,p_2, p_3, p_4\in[1,\infty]$.
Put
$$
\begin{array}{ll}
\ds L^{p_1}_\dbF(\Om;L^{p_2}(t,T;\cX)) =\Big\{\f:(t,T)\times\Om\to
\cX\;\Big|\;\f(\cd)\hb{
is $\mathbb{F}$-adapted and }\\
\hspace{8cm} \dbE\Big(\int_t^T|\f(\tau)|_\cX^{p_2}d\tau\Big)^{\frac{p_1}{p_2}}<\infty\Big\},\\
\ns\ds
 L^{p_4}_\dbF(t,T;L^{p_3}(\Om;\cX)) =\Big\{\f:(t,T)\times\Om\to
\cX\;\Big|\;\f(\cd)\hb{ is $\mathbb{F}$-adapted and
}\\
\hspace{8cm}\int_t^T\Big(\dbE|\f(\tau)|_\cX^{p_3}\Big)^{\frac{p_4}
{p_3}}d\tau<\infty\Big\}.
 \end{array}
 $$


\subsection{Lions Derivative in Infinite Dimension}

We now recall briefly an important notion in mean-field theory: the differentiability
with respect to probability measures. For more details, we refer to \cite{Spille2025, Stannat2025}

Let $\mathcal{P}_2 (H)$ denote the 2-Wasserstein space and $\mathcal W_2(\mu,\nu)=\inf\limits_{\pi\in \Pi(\mu,\nu)}\int_H |x-y|d\pi(x,y)$ the Wasserstein-2-distance, where $\Pi(\mu,\nu)$ is the set of all measures on $H\times H$ with marginals $\mu$ and $\nu$.
  $E$ is an arbitrary real 
Banach space.

Given a 
mapping $\varphi: \mathcal{P}_2 (H)\rightarrow E$, where  the lift 
$$ 
\hat{\varphi} : L^2 (\Omega ; H) \to E, X\mapsto \varphi(\mathcal{L}(X))
$$
is well-defined. Suppose that $\hat{\varphi}$ is Fr\'echet-differentiable at $X$ and denote with 
$D\hat{\varphi}(X)\in \mathcal L(L^2 (\Omega ,H), E)$ its differential. The major difficulty in the infinite dimensional setting is to 
specify to what extent this differential depends on properties of the distribution $\mu := 
\mathcal{L} (X)$ of $X$ only. To this end, the crucial point is to identify $D\hat{\varphi}$ 
with the $E$-valued vector measure $\mu_{D\hat{\varphi}(X)} (A) := D\hat{\varphi} (X) (1_A)$. If its 
regular conditional expectation $\mathbb{E} (\mu_{D\hat{\varphi}(X)}\mid X = x)$ has a 
density w.r.t. $\mu := \mathcal{L}(X)$, then its Radon-Nikodym derivative 
$$ 
\partial_\mu \varphi (\mu )(x) := \frac{d\mu_{D\hat{\varphi}(X)}(x)}{d\mu}
$$ 
is $\mu$-a.s. uniquely determined and satisfies 
$$ 
D\hat{\varphi}(X) (Y) = \mathbb E\big[\partial_\mu \varphi(\mu ) (X) Y\big] , \qq \forall \ Y\in L^2 (\Omega , H). 
$$
In other words, $\partial_\mu \varphi (\mu )$ provides an intrinsic notion of the Lions 
derivative of $\varphi$ in the general case of Banach-space valued mappings. It is shown in the 
theory of vector measures, that the Radom-Nikodym derivative exists if $D\hat{\varphi}(X)$ takes 
values in the space 
$$
\Lambda_2^\mathbb{P} (H) := \{ S\in L (L^2 (\Omega , H, \mathbb P), E)  
\mid \|S\|_{2, \mathbb{P}} < \infty \}
$$
where 
$$
\|S\|_{2, \mathbb{P}} := \sup \{ \sum_{i=1}^n \|S(1_Ax_i)\|_E \mid Y = \sum_{i=1}^n 1_{A_i} x_i , A_i \text{ disj. }, \mathbb{E} (\|Y\|_H^2) \le 1 \}
$$
denotes the $2$-variational norm of $S$. This leads to the following notion of 
$L$-differentiability that we will use in this paper: 

\begin{definition} 
We say that a map $\varphi: \mathcal{P}_2 (H)\rightarrow E$ is $\Lambda$-continuously 
$L$-differentiable (or simply $L$-differentiable) if the lift $\hat{\varphi}$ is continuously Fr\'echet-differentiable, 
its Fr\'echet-differential $D\hat{\varphi}(X) \in \Lambda_2^\mathbb{P} (H, E)$ for all 
$X\in L^2 (\Omega , H)$ and $D\hat{\varphi} : L^2 (\Omega , H) \rightarrow 
\Lambda_2^\mathbb{P}$ is continuous.  
\end{definition} 

\begin{definition}
    We say that $\varphi\in C_b^{1,1}(\mathcal P_2(H); E)$ if for all $\nu\in L^2(\Omega;H)$, there exists a $\mathcal L(\nu)$ modification $\partial \varphi_{\mu}(\mathcal L(\nu))$, denoted by itself, such that $\partial \varphi_{\mu}: \mathcal P_2(H)\times H \mapsto H$ is bounded and Lipschitz continuous, i.e., there exists a constant $C$ such that

    (i) $|\partial_{\mu} \varphi(\mu)(y)|\leq C$, for all $\mu \in \mathcal P_2(H),\ y\in H$;

    (ii) $|\partial_{\mu} \varphi(\mu)(y)-\partial_{\mu} \varphi(\mu^{\prime})(y^{\prime})|\leq C\big(|y-y^{\prime}|+\mathcal W_2(\mu,\mu^{\prime}) \big)$, for all $\mu ,\ \mu^{\prime}\in \mathcal P_2(H),\ y,\ y^{\prime}\in H$.
\end{definition}
\begin{definition}
    We say that $\varphi\in C_b^{2,1}(\mathcal P_2(H); E)$, if $\varphi\in C_b^{1,1}(\mathcal P_2(H); E)$, and 

    (i) $\partial_{\mu} \varphi(\cdot)(y)\in C_b^{1,1}(\mathcal P_2(H); E\times H)$, for all $y\in H$, $\partial_{\mu}^2\varphi: \mathcal P_2(H)\times H\times H\to E\times H\times H $ is bounded and Lipschitz continuous and 

    (ii) $\partial_{\mu}\varphi (\mu)(\cdot): H\to E\times H$ is differentiable, for all $\mu\in \mathcal P_2(H)$, and its derivative $\partial _y\partial _{\mu}\varphi: \mathcal P_2(H)\times H\to E\times H\times H.$
\end{definition}

\begin{definition}
    For $\varphi: H\times \mathcal P_2(H)\to E$, we say $ \varphi \in C_b^{2,2}(H\times \mathcal P_2(H); E)$, if 

    (i) $\varphi(\cdot,\mu)\in C_b^2(H;E)$, for all $\mu\in \mathcal P_2(H)$,

    (ii) $\varphi(x,\cdot)\in C_b^{2,1}(\mathcal P_2(H); E)$ and $\partial_x\varphi (x,\cdot)\in C_b^{1,1}(\mathcal P_2(H); E\times H) $, for all $x\in H$,

    (iii) $\partial_{\mu}\varphi(\cdot, \mu)(\cdot)\in C_b^1(H\times H;E\times H)$, for all $\mu\in \mathcal P_2(H)$,

    (iv) $\varphi $ and all its first and second order derivatives are bounded and Lipschitz continuous.
\end{definition}

Next, we introduce a measurability property that will be needed to ensure the measurability of the coefficients in the adjoint equation (see \cite{Spille2025} for a detailed proof).
\begin{lemma}
\label{lemma:joint measurability}
If $\varphi: H\times \mathcal{P}_2 (H)\to E$ is continuously Fr\'echet differentiable in its 
first component and $\Lambda$-continuously Fréchet differentiable in its second 
component, then there exists a jointly measurable version of
\begin{equation*}
E\times \cP_2(H)\times H\times H \to E,\quad (x,\mu,y,z)\mapsto \partial_\mu \varphi(x,\mu)(y)(z).
\end{equation*}
\end{lemma} 

For notational convenience, we write $\varphi_{\mu}:=\partial_{\mu} \varphi$, $\varphi_{y\mu} 
:= \partial_y\partial_{\mu} \varphi$, $\varphi_{\mu x}:=\partial_{\mu}\partial _x \varphi$ and 
$\varphi_{\mu \mu}:=\partial^2_{\mu} \varphi$ in the rest of the paper. We will always assume sufficient 
regularity on $\varphi$ such that $\partial_x\partial_\mu \varphi = \partial_\mu \partial_x \varphi$. Note 
that interchanging the derivative $\partial_y$ with $\partial_\mu$ is not defined. 


\subsection{Mathematical Setting}

We shall make use of the following Assumptions.

\ss

{\bf (A)} The coefficients $a, \ b,\  f, \ h$ are measurable in all variables, and for all $(t,u)\in [0,T]\times U$,

(i) $a(t,\cdot,\cdot,u)\in C_b^{2,2}(H\times \mathcal P_2(H); H)$,

(ii) $b(t,\cdot,\cdot,u)\in C_b^{2,2}(H\times \mathcal P_2(H); \mathcal L_2^0)$,

(iii) $f(t,\cdot,\cdot,u)\in C_b^{2,2}(H\times \mathcal P_2(H); \mathbb R)$,

(iv) $h(\cdot,\cdot)\in C_b^{2,2}(H\times \mathcal P_2(H); \mathbb R)$.

\ss

\begin{remark}\label{rm.2.1}
   The above assumption is analogous to that in \cite{Buckdahn2016}, where the Pontryagin Maximum Principle (PMP) for a finite-dimensional McKean-Vlasov optimal control problem is studied, and to that in \cite{Yong2020}, which considers the PMP for an optimal control problem with a convex control set. However, the Lions derivative used here is defined in the sense of \cite{Stannat2025}.
\end{remark}

\ss

Now, let $(\overline X(\cdot),\bar u(\cdot))$ be the given optimal pair. Then the following is satisfied:
\begin{equation}\label{system3.1}
\begin{cases}
    d\overline X(t)=A\overline X(t)dt+a(t,\overline X(t),\mathcal L(\overline X(t)), u(t))dt+b(t,\overline X(t),\mathcal L(\overline X(t)),u(t))dW(t),\\
    \hspace{11.5cm}t\in (0,T]\\
    \overline X(0)=\xi \in L^2_{\mathcal F_0}(\Omega;H).
\end{cases}    
\end{equation}
For $\varphi=a,\ b,\ f$, we define 
\begin{equation*}
\begin{cases}
    \varphi(t)\triangleq \varphi (t,\overline X(t),\mathcal L(\overline X(t)),\bar u(t)),\\
    \varphi_x(t)\triangleq \varphi_x(t,\overline X(t),\mathcal L(\overline X(t)),\bar u(t)),\  \varphi_{xx}(t)\triangleq \varphi_{xx}(t,\overline X(t),\mathcal L(\overline X(t)),\bar u(t)),\\
    \delta \varphi(t)\triangleq \varphi (t,\overline X(t),\mathcal L(\overline X(t)), u(t))-\varphi (t,\overline X(t),\mathcal L(\overline X(t)), \bar u(t)),\\
    \delta \varphi_x(t)\triangleq \varphi_x (t,\overline X(t), \mathcal L(\overline X(t)),u(t))-\varphi_x (t,\overline X(t), \mathcal L(\overline X(t)),\bar u(t)),\\
    \delta \varphi_{xx}(t)\triangleq \varphi_{xx} (t,\overline X(t), \mathcal L(\overline X(t)),u(t))-\varphi_{xx} (t,\overline X(t),\mathcal L(\overline X(t)), \bar u(t)),\\
    \varphi_{\mu}(t)(y)\triangleq  \varphi_{\mu} (t,\overline X(t),\mathcal L(\overline X(t)), \bar u(t))(y), \ y\in H,\\
    \varphi_{y\mu}(t)(y): (z_1,z_2)\mapsto \varphi_{y\mu}(t)(y)[z_1,z_2], \ y,\  z_1,\ z_2\in H,\\
    \varphi_{\mu x}(t)(y): (z_1,z_2)\mapsto \varphi_{\mu x}(t)(y)[z_1,z_2], \ y,\  z_1,\ z_2\in H,\\
    \varphi_{\mu\mu}(t)(y_1,y_2): (z_1,z_2)\mapsto \varphi_{\mu\mu}(t)(y_1,y_2)[z_1,z_2], \ y_1,\ y_2, \ z_1,\ z_2\in H,\\
\end{cases}
\end{equation*}
We denote by $(\widetilde{\Omega}, \widetilde{\mathcal F},\widetilde{\mathbb P})$ an independent copy of the space 
$(\Omega, \mathcal F, \mathbb P)$, and denote with $\widetilde{\mathbb{E}}$ the corresponding expectation. 
For any random variable $X\in L^2(\Omega, \mathcal F,\mathbb P; H)$, we denote by $\widetilde{\overline{X}}$ the independent 
copy of $X$ on $(\widetilde{\Omega}, \widetilde{\mathcal F},\widetilde{\mathbb P})$. Let $(\tilde{\overline{u}}, \widetilde{\overline{X}})$ be an 
independent copy of $(\bar u,\overline X)$, so that $\mathcal L(\overline X(t))=\mathcal L(\widetilde{\overline{X}}(t)),\ 
t\in [0,T]$. 

We denote 
\begin{equation*}
    \tilde {\varphi}(t)\triangleq\varphi (t,\widetilde{\overline{X}}(t),\mathcal L(\widetilde{\overline{X}}(t)),\widetilde{\overline{u}}(t)).
\end{equation*}
%

\subsection{Main Result}
We are now ready to introduce the two adjoint equations and to establish our main result, namely the Pontryagin maximum principle.

The first-order adjoint equation is 
\begin{equation}\label{1-adeq}
\begin{cases}
    dp(t)=-\big\{ A^*p(t)+a_x(t)^*p(t)+b_x(t)^*q(t)- f_x(t)\big\}dt\\
    \hspace{1.5cm}-\big(\widetilde{\mathbb E}\big[ \tilde{a}_{\mu}(t)(\overline{X}(t))^* \tilde p(t)\big]+\widetilde{\mathbb E}\big[\tilde{b}_{\mu}(t)(\overline{X}(t))^*\tilde q(t)\big]  
    - \widetilde{\mathbb E}\big[ \tilde{f}_{\mu}(t)(\overline{X}(t))\big] \big)dt\\
    \hspace{1.5cm}+ q(t)dW(t),\ t\in [0,T],\\
    p(T)=-h_x(\overline X(T), \mathcal L(\overline X(T)))-\widetilde{\mathbb E}\big[ h_{\mu}(\widetilde{\overline{X}}(T), \mathcal L(\widetilde{\overline{X}}(T)))(\overline{X}(T))\big].
\end{cases}
\end{equation}
where $\tilde{a}_\mu (t) (X(t)) = a_\mu (t, \widetilde{\overline{X}}(t), \mathcal{L} (\widetilde{\overline{X}}_t) , 
\tilde{\overline{u}}_t)(X(t))$ and similar for $\tilde{b}_\mu$ and $\tilde{f}_\mu$. Note that \eqref{1-adeq} is a 
McKean–Vlasov BSEE. From the discussion in Subsection \ref{subsec-trans}, under 
Assumption {\bf (A)}, it admits a unique transposition solution.

An important device in the non-convex control case is the introduction of a second-order 
adjoint equation, inspired by \cite{Peng1990}. To simplify notations, define the 
Hamiltonian 
$$ 
\mathbb H(t,x,\mu,u,y,z):=\langle y,a(t,x,\mu,u)\rangle+\langle z,b(t,x,\mu,u)\rangle 
-f(t,x,\mu,u)
$$ 
for $(t,x,\mu,u,y,z)\in [0,T]\times  H\times \mathcal P_2(H)\times U\times H \times \mathcal L_2^0$, so that 
\begin{equation}\label{1-max}
    \mathbb H(t,\overline X(t),\mathcal L(\overline X(t))),\bar u(t),p(t),q(t)) 
    = \max_{u\in U}\mathbb H(t,\overline X(t),\mathcal L(\overline X(t))),u,p(t),q(t))
\end{equation}
The second-order adjoint equation is then given as follows:
{\small 
\begin{equation}
\label{ad-eq2}
\begin{cases}
   dP(t) & = - \Big\{ \big(A^*+a^*_x(t)\big)P(t)+ P(t)\big(A+a_x(t)\big) \\
   &  \qquad + b^*_x(t) P(t) 
     b_x(t) + b^*_x(t)  Q(t)  +Q(t) b_x(t)  \\
 & \qquad+\mathbb{H}_{xx}(t)+\widetilde{\mathbb E}\big[ \widetilde{\mathbb H}_{y\mu}(t)(\overline{X}(t))\big]\Big\} dt  
   + Q(t)dW(t), \ \  t\in [0,T),\\
P(T) & = - h_{xx}\left(\overline{X}(T), \mathcal L\left( \overline{X}(T)\right)\right) 
   - \widetilde{\mathbb{E}}\big[h_{y\mu} \big( \widetilde{\overline{X}}_T, \mathcal L ( \widetilde{\overline{X}}(T) ) \big)(\overline{X}(T))\big].
\end{cases}
\end{equation}
}

Note that unlike \cite{Buckdahn2016}, the second-order adjoint equation \eqref{ad-eq2} does not include the mean-field diffusion derivative term $\widetilde{\mathbb E}\big[\tilde{b}_\mu (t)(\overline{X}(t))\big]$. The reason is that, in the infinite dimensional setting, the well-posedness theory for relaxed transposition solution cannot be applied directly in the presence of such mean-field terms. Extending the existing framework to accommodate these terms requires excessively lengthy discussion and significant technical complications. Therefore, we omit these terms from the second order adjoint equation and apply the well-posedness result given in \cite{Lu2021} to our framework. 

With the well-posedness of both the first- and second-order adjoint equations (to be established and presented in Section \ref{sec-trans}), we are now ready to state the following stochastic maximum principle. 

\begin{theorem}\label{thm4.1}
    (Stochastic Maximum Principle) Suppose that \textbf{{\bf (A)}} hold. Let $(\overline X(\cdot), \bar u(\cdot))$ be an optimal solution to Problem \textbf{(S)}. Let $(p(\cdot),q(\cdot))$ be the transposition solution to \eqref{1-adeq}, $(P(\cdot),Q^{(\cdot)},\hat Q^{(\cdot)})$ be the relaxed transposition solution to \eqref{ad-eq2}. Then, for all $u\in U$, and $a.e. t\in [0,T]$, it holds $\mathbb P$-almost surely that 
\begin{eqnarray*}
    0& \leq & \mathbb H(t,\overline X(t), \mathcal L (\overline X(t)), \bar u(t), p(t),P(t))- \mathbb H(t,\overline X(t), \mathcal L (\overline X(t)), u, p(t),P(t))\\
    & & -\frac{1}{2} \big\langle P(t)\big (b(t)-b(t,\overline X(t),\mathcal L(\overline X(t)),u)\big ),b(t)-b(t,\overline X(t),\mathcal L(\overline X(t)),u)\big\rangle_{\mathcal L_2^0}.
\end{eqnarray*}

\end{theorem}

\ss


\section{Variational Equations}\label{subsec-VE}
In this section, we present the variational equations which
are crucial in establishing the Pontryagin maximum principle. Since the control set $U$ is not necessarily convex, we shall use the so-called spike variation. For any fixed $u(\cdot)\in \mathcal U[0,T]$ and $\varepsilon>0$, define 
\begin{equation}\label{eq-4.1.2}
    u^{\varepsilon}(t)=\begin{cases} \bar u(t),\hspace{1cm} t\in [0,T]\setminus E_{\varepsilon},\\
    u(t), \hspace{1cm} t\in E_{\varepsilon},       
    \end{cases}
\end{equation}
where $E_{\varepsilon}\subset [0,T]$ is a measurable set with $|E_{\varepsilon}|=\varepsilon$. Let $(x^{\varepsilon}(\cdot),u^{\varepsilon}(\cdot))$ satisfy the following:
{\small
\begin{equation}\label{system3.1a}
\begin{cases}
    dX^{\varepsilon}(t)=AX^{\varepsilon}(t)dt+a(t,X^{\varepsilon}(t),\mathcal L(X^{\varepsilon}(t)), u^{\varepsilon}(t))dt+b(t,X^{\varepsilon}(t),\mathcal L(X^{\varepsilon}(t)),u^{\varepsilon}(t))dW(t),\\
    \hspace{11.8cm}t\in (0,T]\\
    X^{\varepsilon}(0)=\xi \in L^2_{\mathcal F_0}(\Omega;H).
\end{cases}    
\end{equation}
}
In classical stochastic optimal control problems, two variational equations are typically introduced to derive the second-order expansion of the cost functional. Establishing the corresponding order estimates of these variational equations plays a crucial role in proving the stochastic maximum principle.
However, in order to accommodate the framework of relaxed transposition solution, it is necessary to reformulate the variational systems. Rather than working directly with the classical first- and second-order variational equations, we consider the following rearranged variational equations: 
the first variational equation:
\begin{equation}\label{eq-4.4}
\begin{cases}
    dy^{\varepsilon}(t)=Ay^{\varepsilon}(t)dt+a_x(t)y^{\varepsilon}(t)dt + \Big(b_x(t)   y^{\varepsilon}(t)+\delta b(t)\chi_{E_{\varepsilon}}(t)\Big)dW(t)\\   
    y^{\varepsilon}(0)=0,
\end{cases}
\end{equation}
and the second variational equation:
{\small 
\begin{equation}\label{eq-4.5}
\begin{cases}
    dz^{\varepsilon}(t)=Az^{\varepsilon} (t) dt+\Big(a_x(t)z^{\varepsilon}(t)+\widetilde{\mathbb E}\big[a_{\mu}(t)(\widetilde{\overline{X}}(t))\tilde z^{\varepsilon}(t)\big]+ \frac{1}{2} a_{xx}(t) [y^{\varepsilon}(t), y^\varepsilon (t)] \\
    \hspace{3.5cm}+\widetilde{\mathbb E}\big[a_{\mu}(t)(\widetilde{\overline{X}}(t))\tilde y^{\varepsilon}(t)\big]+ \frac{1}{2}\widetilde{\mathbb E}\big[ a_{y\mu}(t)(\widetilde{\overline{X}}(t)) [\tilde{y}^{\varepsilon}(t), \tilde{y}^\varepsilon (t)]\big]\Big)dt\\
    \hspace{1.5cm} + \Big(b_x(t)   z^{\varepsilon}(t)+\widetilde{\mathbb E}\big[b_{\mu}(t)(\widetilde{\overline{X}}(t))\tilde z^{\varepsilon}(t)\big]+\frac{1}{2}b_{xx}(t)[y^{\varepsilon}(t), y^\varepsilon (t)]\\
    \hspace{3.5cm} +\widetilde{\mathbb E}\big[b_{\mu}(t)(\widetilde{\overline{X}}(t))\tilde y^{\varepsilon}(t)\big]+\frac{1}{2}\widetilde{\mathbb E}\big[b_{y\mu}(t)(\widetilde{\overline{X}}(t))[\tilde {y}^{\varepsilon}(t), \tilde{y}^\varepsilon (t) ]\big]\Big)dW(t)\\ 
    \hspace{1.5cm} + \Big(\delta a(t)+\delta a_x(t)y^{\varepsilon}(t)+\widetilde{\mathbb E}\big[\delta a_{\mu}(t)(\widetilde{\overline{X}}(t))\tilde y^{\varepsilon}(t)\big]\Big)\chi_{E_{\varepsilon}}(t)dt\\
    \hspace{1.5cm}+\Big(\delta b_x(t)y^{\varepsilon}(t)+\widetilde{\mathbb E}\big[\delta b_{\mu}(t)(\widetilde{\overline{X}}(t))\tilde y^{\varepsilon}(t)\big] \Big)\chi_{E_{\varepsilon}}(t)dW(t)\\
    z^{\varepsilon}(0)=0, 
\end{cases}   
\end{equation}
}
with $y^{\varepsilon},\ z^{\varepsilon}\in C_{\mathbb F}(0,T;L^p(\Omega;H))$, for all $p\ge 1$, the mild solutions to \eqref{eq-4.4} and \eqref{eq-4.5}. 

Note that in with the rearranged variational equations, we no longer have $\mathbb E \sup\limits_{0 \leq t \leq T}|z^{\varepsilon}(t)|^{2k}=O(\varepsilon^{2k})$, but $\mathbb E \sup\limits_{0 \leq t \leq T}|z^{\varepsilon}(t)|^{2k}=o(\varepsilon^{k})$. As we will show later, despite the weaker estimate for $z^{\varepsilon}(t)$, the estimate $X^{\varepsilon}(t)-\overline X(t)-y^{\varepsilon}(t)-z^{\varepsilon}(t)=o(\varepsilon)$, as $\varepsilon\to 0$, still holds. This is sufficient to obtaining the desired second-order expansion of the cost functional, for establishing the stochastic maximum principle.

To do so, we establish some fundamental estimates that will play crucial roles in our discussion. In this paper, $\cC$
is a generic constant which may vary from line to line.

\begin{proposition}\label{prop4.3}
Under Assumption {\bf (A)}, for any $k\ge 1$, and $\varepsilon>0$, the following estimates hold:
\begin{equation}\label{eq-4.6}
    \mathbb E \sup\limits_{0 \leq t \leq T}|X^{\varepsilon}(t)-\overline X(t)|^{2k}=O(\varepsilon^{k}),
\end{equation}
\begin{equation}\label{eq-4.7}
     \mathbb E \sup\limits_{0 \leq t \leq T}|y^{\varepsilon}(t)|^{2k}=O(\varepsilon^{k}),
\end{equation}
\end{proposition}

{\it Proof}.
Let's first proof \eqref{eq-4.6}.
Let $\xi^{\varepsilon}(t)=X^{\varepsilon}(t)-\overline X(t)$, then 
\begin{equation}\label{eq-4.10}
\begin{cases}
    d\xi^{\varepsilon}(t) = A\xi^{\varepsilon}(t)dt + \Big( a^{\theta}_x(t)\xi^{\varepsilon}(t) 
    +\widetilde{\mathbb E}\big[ a^{\theta}_{\mu}(t)\tilde{\xi}^{\varepsilon}(t)\big] 
    + \delta a(t) \chi_{E_{\varepsilon}}(t)\Big)dt\\
    \hspace{1.5cm} + \Big( b^{\theta}_x(t)\xi^{\varepsilon}(t)+\widetilde{\mathbb E} 
    \big[ b^{\theta}_{\mu}(t)\tilde{\xi}^{\varepsilon}(t)\big] 
    + \delta b(t)\chi_{E_{\varepsilon}}(t)\Big)dW(t)\\
    \xi^{\varepsilon}(0)=0.
\end{cases}
\end{equation}
where for $\phi = a,\ b$,
\begin{equation}\label{eq-4.11}
\begin{cases}
   \phi^{\theta}_x(t)\triangleq \int_0^1 \phi_x(t,\overline X(t)+\theta\xi^{\varepsilon}(t), \mathcal L(\overline X(t)+\theta\xi^{\varepsilon}(t)) ,u^{\varepsilon}(t))d\theta,\\
   \phi^{\theta}_{\mu}(t)\triangleq \int_0^1 \phi_{\mu}(t,\overline X(t)+\theta\xi^{\varepsilon}(t), \mathcal L(\overline X(t)+\theta\xi^{\varepsilon}(t)) ,u^{\varepsilon}(t))(\widetilde{\overline{X}}(t)+\theta\tilde{\xi}^{\varepsilon}(t))d\theta.
\end{cases}
\end{equation}
By Assumption {\bf (A)}, $\phi^{\theta}_x$ and $\phi^{\theta}_{\mu}$ are bounded. From \cite[Chapter 3]{Lu2021} the definition of mild solution to \eqref{eq-4.10}, and by Burkholder-Davis-Gundy type inequality (see \cite[Chapter 3, Theorem 3.18]{Lu2021}, together with the boundedness of $\phi^{\theta}_x$ and $\phi^{\theta}_{\mu}$, we have 
for all $0\le t\le T$ 
\begin{eqnarray*}
  & &  \mathbb E \sup\limits_{0 \leq s \leq t}|\xi^{\varepsilon}(s)|^{2k}\\
    & \leq  &  \cC \mathbb E \sup\limits_{0 \leq s \leq t}\Big|\int_0^s S(s-r) \Big[a^{\theta}_x(r)\xi^{\varepsilon}(r) 
    +\widetilde{\mathbb E}\big[ a^{\theta}_{\mu}(r)\tilde{\xi}^{\varepsilon}(r)\big] 
    + \delta a(r) \chi_{E_{\varepsilon}}(r)\Big]dr\Big|^{2k}\\
    & & +\cC \mathbb E \sup\limits_{0 \leq s \leq t}\Big|\int_0^s S(s-r) 
    \Big[b^{\theta}_x(r)\xi^{\varepsilon}(r)+\widetilde{\mathbb E} 
    \big[ b^{\theta}_{\mu}(r)\tilde{\xi}^{\varepsilon}(r)\big] 
    + \delta b(r)\chi_{E_{\varepsilon}}(r)\Big]dW(r)\Big|^{2k}\\
    & \leq & \cC\mathbb E \Big(\int_0^t \big|a^{\theta}_x(r)\xi^{\varepsilon}(r) 
    +\widetilde{\mathbb E}\big[ a^{\theta}_{\mu}(r)\tilde{\xi}^{\varepsilon}(r)\big] 
    + \delta a(r) \chi_{E_{\varepsilon}}(r)\big|dr\Big)^{2k}\\   
    & & + \cC \mathbb E \Big(\int_0^t \big|b^{\theta}_x(r)\xi^{\varepsilon}(r)+\widetilde{\mathbb E} 
    \big[ b^{\theta}_{\mu}(r)\tilde{\xi}^{\varepsilon}(r)\big] 
    + \delta b(r)\chi_{E_{\varepsilon}}(r)\big|_{\mathcal L_2^0}^2 dr\Big)^{k}\\
    & \leq & \cC \int_0^t \mathbb E |\xi^{\varepsilon}(r)|^{2k}dr+ \cC \mathbb E\Big(\int_0^T |\delta a(r)\chi_{E_{\varepsilon}}(r)|dr \Big)^{2k}+\cC \mathbb E \Big(\int_0^T |\delta b(r)\chi_{E_{\varepsilon}}(r)\big|_{\mathcal L_2^0}^2 dr\Big)^k,
\end{eqnarray*}
therefore, by Gronwall's inequality,

\begin{equation*}
\begin{aligned}
\mathbb E & \sup\limits_{0 \leq t \leq T}|X^{\varepsilon}(t)-\overline X(t)|^{2k} 
      =  \mathbb E \sup\limits_{0 \leq t \leq T}|\xi^{\varepsilon}(t)|^{2k}\\
     & \leq \cC \Big\{ \mathbb{E}  \Big[ \Big(\int_0^T  |\delta a(t)\chi_{E_{\varepsilon}}(t)|dt \Big)^{2k} \Big] 
     + \mathbb{E}\Big[ \Big(\int_0^T |\delta b(t)\chi_{E_{\varepsilon}}(t)|^{2}_{\mathcal{L}_2^0} dt\Big)^k \Big]\Big\}\\
     & \leq  \cC (\varepsilon^{2k}+\varepsilon^k)\leq \cC \varepsilon^k.
\end{aligned}
\end{equation*} 
This proves \eqref{eq-4.6}. Similarly, we can prove \eqref{eq-4.7}.

\endpf

\ss

As discussed in \cite{Peng1990}, our ultimate goal is to derive a second-order Taylor expansion of the state equation in the following sense:
\begin{equation*}
     X^{\varepsilon}(t)=\overline X(t)-y^{\varepsilon}(t)-z^{\varepsilon}(t)+o(\varepsilon).
\end{equation*}
The analysis in \cite[Proposition 4.3]{Buckdahn2016} shows that establishing suitable estimates for the variational equations is of crucial importance. However, the method developed in \cite{Buckdahn2016} cannot be directly extended to our infinite-dimensional framework, since the stochastic exponential and its inverse are generally not well-defined in this setting. 

Moreover, unlike thie approach in \cite{Guatteri2025}, where the BSEE for $\phi(t)$ is operator-valued and its well-posedness relies on the ansatz that it is essentially finite dimensional, such a method cannot be extended to our genuinely infinite-dimensional setting due to the lack of well-posedness. In contrast, the BSEE \eqref{eq--3.12a} arising in our framework is H-valued and therefore requires a different estimation technique. 

Nevertheless, although the approach of \cite[Proposition 3.3]{Guatteri2025}, cannot be directly extended, it provides a new viewpoint. Motivated by this perspective, we develop a new approach adapted to the infinite-dimensional framework and get the following estimates, which constitute a key step in the proof of Proposition \ref{prop4.4}. In addition, our method can also be applied to the special case of delay equation considered in \cite{Guatteri2025}.

\begin{proposition}\label{prop4.5-1}
    Suppose Assumption {\bf (A)} holds and let $\phi(\cdot)\in L_{\mathbb F}^2(0,T;\mathcal L(H)))$, then for any $\varepsilon>0$, 
\begin{equation}\label{eq---4.20}
   \int_0^T \big|\mathbb E[\phi(s)y^{\varepsilon}(s)]\big|^2 ds = o(\varepsilon ), 
\end{equation}
\begin{equation}\label{eq--3.11-1}
    \big|\mathbb E[\phi(T)y^{\varepsilon}(T)]\big|^2 = o(\varepsilon ). 
\end{equation}
\end{proposition}

{\it Proof.} For any fixed $s\in [0,T]$ and $h\in H$, consider the $H$-valued BSEE with 
terminal time $s$:
{\small
\begin{equation}\label{eq--3.12a}
\begin{cases}
    -d\Phi^{(s)}_h (t) 
    = \Big(A^*\Phi^{(s)}_h (t) + a^*_x(t)\Phi^{(s)}_h (t) 
    +b^*_x(t)\Psi^{(s)}_h (t) 
    \Big) dt - \Psi^{(s)}_h (t) dW(t), 
    \q t\in [0,s),\\
    \Phi^{(s)}_h (s) = \phi^*(s)h,
\end{cases}
\end{equation} 
}
By \cite[Theorem 4.16]{Lu2021} it has a unique transposition solution $(\Phi^{(s)}_h (\cdot ),\Psi^{(s)}_h (\cdot)) 
\in D_\mathbb{F} ([0,s] ;$
$ L^2(\Omega; H)) \times L^2_\mathbb{F} (0,s; \mathcal{L}_2^0)$, satisfying 
\begin{equation} 
\label{eq-4.4a}
\int_0^s \mathbb E  \big[ \big|\Phi^{(s)}_h (r)\big|^2\big] +  \mathbb E \big[\big|\Psi^{(s)}_h (r) 
\big|_{\mathcal L_2^0}^2\big] dr \le \cC \mathbb{E} \big[ |\phi^* (s) h|^2 \big]\, . 
\end{equation} 
With $y^{\varepsilon}$ being the mild solution of \eqref{eq-4.4}, it follows that 
\begin{equation} 
\label{eq--3.12}
\begin{aligned} 
\mathbb E \big[ \langle \phi(s) y^{\varepsilon}(s),h \rangle \big]
    & = \mathbb E \big[ \langle \phi^* (s) h, y^{\varepsilon}(s)\rangle \big] \\
    & =   \mathbb E\Big[\int_0^s  \langle \Psi^{(s)}_h (r), \delta b(r)\chi_{E_{\varepsilon}}(r)\rangle_{\mathcal{L}_2^0}  dr\Big]. 
\end{aligned} 
\end{equation}
By Fubini theorem and Cauchy-Schwarz inequality, and noting that $b$ is bounded, we get 
\begin{equation} 
\label{eq--3.13}
\begin{aligned} 
    & \mathbb E\Big[\int_0^s  \langle \Psi^{(s)}_h (r), \delta b(r)\chi_{E_{\varepsilon}}(r)\rangle_{\mathcal{L}_2^0}  dr\Big] \\ 
     & \leq \cC \int_0^s \mathbb E \big[ \big|\Psi^{(s)}_h (r)\big|_{\mathcal L_2^0} \big] \chi_{E_\varepsilon} (r) dr \\
     & \leq \cC \varepsilon^{\frac 12} \Big( \int_0^s \mathbb E \big[  \big|\Psi^{(s)}_h (r)\big|_{\mathcal L_2^0}^2 \big] \chi_{E_\varepsilon} (r) dr \Big)^{\frac 12} . 
\end{aligned} 
\end{equation}

We now choose $h = h^{(s)} := \frac{h_0}{|h_0|}$, where $h_0 = \mathbb E \big[ \phi(s) y^{\varepsilon}(s) \big]$ if $h_0 \neq 0$ and $h^{(s)} = 0$ otherwise. Then \eqref{eq--3.12} and \eqref{eq--3.13} imply that 
\begin{equation}
\label{eq--3.13a} 
\begin{aligned} 
\big|\mathbb E \big[ \phi(s) y^{\varepsilon}(s) \big] \big|^2 
& = \big|\mathbb{E} \big[ \langle \phi (s) y^\varepsilon (s) , h^{(s)} \rangle \big]\big|^2 \\
& \leq \cC \varepsilon \int_0^s \mathbb E \big[ \big|\Psi^{(s)}_{h^{(s)}} (r)\big|_{\mathcal L_2^0}^2 \big] \chi_{E_\varepsilon} (r) dr . 
\end{aligned} 
\end{equation}
Integrating up the last inequality w.r.t. $s$ then yields 
\begin{eqnarray}\label{eq--3.13b} 
\int_0^T \big| \mathbb E \big[ \phi(s) y^{\varepsilon}(s) \big] \big|^2 ds 
  \leq \cC \epsilon \int_0^T \int_0^s  \mathbb E \big[ \big|\Psi^{(s)}_{h^{(s)}} (r)\big|_{\mathcal L_2^0}^2 \big] \chi_{E_\varepsilon} (r) dr ds.  
\end{eqnarray}
Using \eqref{eq-4.4a} and $|h^{(s)}|\le 1$ for all $s$, we have that 
\begin{eqnarray}\label{eq--3.17}
\int_0^T \int_0^s \mathbb E \big[ \big|\Psi^{(s)}_{h^{(s)}} (r)\big|_{\mathcal L_2^0}^2 \big] dr ds 
\le \cC \int_0^T \mathbb{E}\big[ |\phi^* (s)|^2_{\mathcal{L} (H)}\big] ds  
 < \infty .  
\end{eqnarray}
By Lebesgue's dominated convergence this implies that 
\begin{eqnarray}\label{eq--3.13c} 
\rho (\varepsilon ) := \int_0^T \int_0^s \mathbb E \big[  \big|\Psi^{(s)}_{h^{(s)}} (r)\big|_{\mathcal L_2^0}^2 \big] \chi_{E_\varepsilon} (r) dr ds
\rightarrow 0 \text{ as }\varepsilon\to 0 , 
\end{eqnarray}
which together with \eqref{eq--3.13b} yields \eqref{eq---4.20}. 

\eqref{eq--3.11-1} comes inmediately from \eqref{eq--3.13a} by taking $s=T$.

\endpf

From the Proposition \ref{prop4.5-1}, we get immediately the following Corollary.

\begin{corollary}\label{cor3.1}
Let $(\overline{\Omega} , \overline{\mathcal{F}}, \overline{\mathbb{P}})$ be an arbitrary probability space, 
and let $\phi \in L^2 ([0, T]$
$\times\overline{\Omega}\times\Omega , \mathcal{B} ([0,T]) \otimes 
\overline{\mathcal{F}} \otimes \mathcal{F}, ds\otimes \overline{P}\otimes P; \mathcal{L}(H))$ be such  
that $\phi (\cdot , \overline{\omega} , \cdot)\in L_{\overline{\mathcal{F}}\otimes\mathbb{F}}^2(0,T;
\mathcal L(H))$ $\overline{P}$-a.s. Similarly, let $\psi\in L^2 (\overline{\Omega}\times\Omega , 
\overline{\mathbb{F}}\otimes\mathcal{F}_T , \overline{\mathbb{P}}\otimes \mathbb{P} ;\mathcal L(H)))$. Then 
for any $\varepsilon>0$,
 \begin{equation} 
 \label{eq--3.21}     
  \int_0^T \overline{\mathbb E} \big[ \big| \mathbb E \big[ \phi( s, \overline{\omega}, \cdot ) 
  y^{\varepsilon}(s)\big]\big|^2 \big] ds = o(\varepsilon)
    \end{equation}
and 
\begin{equation}\label{eq--3.21a}     
   \overline{\mathbb E} \big[ \big| \mathbb E \big[\psi( \overline{\omega}, \cdot) y^{\varepsilon}(T)\big]\big|^2 \big] = o(\varepsilon). 
    \end{equation}
\end{corollary}

{\it Proof}.
For fixed $\overline{\omega}\in\overline{\Omega}$, consider the process 
\begin{equation*}
    \phi^{\overline{\omega}}(s,\omega)\triangleq\phi(s,\overline{\omega}, \omega).
\end{equation*}
Then applying Proposition \ref{prop4.5-1} on the probability space $(\overline{\Omega} , \overline{\mathcal{F}}, \overline{\mathbb{P}})$ with test function $\phi^{\overline{\omega}}(s,\cdot)$ gives 
\begin{equation*}
    \int_0^T \big| \mathbb E\big[\phi(s,\overline{\omega}, \omega)\overline{y}^{\varepsilon}(s)\big]\big|^2ds=o(\varepsilon),\qquad a.a.\ \overline{\omega}\in \overline{\Omega}.
\end{equation*}
Set $G_{\varepsilon}(\overline{\omega})\triangleq  \int_0^T \big| \widetilde{\mathbb E}\big[\phi(s,\overline{\omega}, \omega)\overline{y}^{\varepsilon}(s)\big]\big|^2ds$, then $G_{\varepsilon}(\overline{\omega})=o(\varepsilon)$ for $a.a.\ \omega\in \Omega$. Thus, noting Fubini theorem, we get
\begin{equation}\label{eq--3.20}
    \overline{\mathbb E}\big[G_{\varepsilon}(\overline{\omega})\big]=\overline{\mathbb E}\Big[ \int_0^T \big| \mathbb E\big[\phi(s, \overline{\omega},\omega) \overline{y}^{\varepsilon}(s)\big]\big|^2ds   \Big]= \int_0^T \overline{\mathbb E}\big| \mathbb E\big[\phi(s,\overline{\omega}, \omega)\overline{y}^{\varepsilon}(s)\big]\big|^2ds.
\end{equation}
From \eqref{eq--3.13b} and \eqref{eq--3.17}, 
\begin{equation*}
    0\leq \frac{G_{\varepsilon}(\overline{\omega})}{\varepsilon}\leq  \cC \int_0^T \overline{\mathbb{E}}\mathbb E\big[ |\phi^* (s,\overline{\omega},\omega)|^2_{\mathcal{L} (H)}\big] ds< \infty,
\end{equation*}
thus dominated convergence gives 
\begin{equation*}
    \frac{\overline{\mathbb E}\big[G_{\varepsilon}(\overline{\omega})\big]}{\varepsilon}\to 0,\ \ i.e., \ \overline{\mathbb E}\big[G_{\varepsilon}(\overline{\omega})\big]=o(\varepsilon).
\end{equation*}
This combining with \eqref{eq--3.20} yields 
\begin{equation*}
    \int_0^T \overline{\mathbb E}\big| \mathbb E\big[\phi(s,\overline{\omega}, \omega)\overline{y}^{\varepsilon}(s)\big]\big|^2ds=o(\varepsilon).
\end{equation*}

Similarly, we get \eqref{eq--3.21a}.

\endpf

\ss

Following the same arguments as in Proposition \ref{prop4.5-1} and Corollary \ref{cor3.1}, we obtain similar estimates for $\xi^{\e}$.

\begin{corollary}\label{cor3.2}
    With the assumption in Corollary \ref{cor3.1}, we have for any $\e>0$,
    \begin{equation} 
 \label{eq--3.21b}     
  \int_0^T \overline{\mathbb E} \big[ \big| \mathbb E \big[ \phi( s, \overline{\omega}, \cdot ) 
  \xi^{\varepsilon}(s)\big]\big|^2 \big] ds = o(\varepsilon)
    \end{equation}
and 
\begin{equation}\label{eq--3.21ab}     
   \overline{\mathbb E} \big[ \big| \mathbb E \big[\psi( \overline{\omega}, \cdot) \xi^{\varepsilon}(T)\big]\big|^2 \big] = o(\varepsilon). 
    \end{equation}
\end{corollary}

{\it Proof}. The proof of Corollary \ref{cor3.2} follows immediately by applying same arguments in Proposition \ref{prop4.5-1} and Corollary \ref{cor3.1}, with the BSEE \eqref{eq--3.12a} replaced by 
{\small
\begin{equation}\label{eq--3.12ab}
\begin{cases}
    -d\Phi^{(s)}_h (t) 
    = \Big(A^*\Phi^{(s)}_h (t) + a^{\theta *}_x(t)\Phi^{(s)}_h (t) 
    +\widetilde{\mathbb E}\big[ a^{\theta *}_{\mu}(t)\tilde{\Phi}^{(s)}_h (t)\big]\\
    \hspace{2.5cm}
    +b^{\theta *}_x(t)\Psi^{(s)}_h (t) + \widetilde{\mathbb E}\big[ b^{\theta *}_{\mu}(t)\tilde{\Psi}^{(s)}_h (t)\big]
    \Big) dt - \Psi^{(s)}_h (t) dW(t), 
    \q t\in [0,s),\\
    \Phi^{(s)}_h (s) = \phi^*(s)h,
\end{cases}
\end{equation} 
}
where $a^{\theta *}_x(t),\ b^{\theta *}_x(t),\ a^{\theta *}_{\mu}(t),\mbox{ and } b^{\theta *}_{\mu}(t)$ are the adjoints of $a^{\theta}_x(t),\ b^{\theta}_x(t) ,\ a^{\theta}_{\mu}(t),\mbox{ and } b^{\theta}_{\mu}(t)$ respectively. By Theorem \ref{thm1}, it admits a unique transposition solution. 

\endpf

With Corollary \ref{cor3.1} and \ref{cor3.2}, we have the following estimates.

\begin{proposition}\label{prop-3.2}
    Under Assumption {\bf (A)}, for any $\varepsilon>0$, the following estimates hold:
\begin{equation}\label{eq-4.8}
     \mathbb E \sup\limits_{0 \leq t \leq T}|z^{\varepsilon}(t)|^{2}=o(\varepsilon),
\end{equation}
\begin{equation}\label{eq-4.9}
      \mathbb E \sup\limits_{0 \leq t \leq T}|X^{\varepsilon}(t)-\overline X(t)-y^{\varepsilon}(t)|^{2}=O(\varepsilon).
\end{equation}
\end{proposition}

{\it Proof}.
First, let's prove \eqref{eq-4.8}. From \eqref{eq-4.7} and Corollary \ref{cor3.1} and the boundedness of first and second order derivatives of $a$ and $b$, together with Burkholder-Davis-Gundy inequality and Gronwall's inequality, we have that
\begin{equation*}
\begin{aligned} 
\mathbb E & \sup\limits_{0 \leq t \leq T}|z^{\varepsilon}(t)|^{2} \\
    & \leq \cC \Big\{ \mathbb{E} \Big[ \Big(\int_0^T |a_{xx}(t) [y^{\varepsilon}(t), y^\varepsilon (t)]| 
    + \tilde{\mathbb{E}} \big[ |a_{y\mu}(t)(\widetilde{\overline{X}}(t))[\tilde{y}^{\varepsilon}(t), \tilde{y}^\varepsilon (t)]|\big]  \\
    & \qquad\qquad  +\widetilde{\mathbb E}\big[|a_{\mu}(t)(\widetilde{\overline{X}}(t))\tilde y^{\varepsilon}(t)|\big]\\
    & \qquad\qquad + \big\{|\delta a(t)|+|\delta a_x(t)y^{\varepsilon}(t)| + \widetilde{\mathbb E}\big[|\delta a_{\mu}(t)(\widetilde{\overline{X}}(t))\tilde y^{\varepsilon}(t)|\big]\big\}\chi_{E_{\varepsilon}}(t) dt \Big)^{2}\Big]\\
    &  + \mathbb{E}  \Big(\int_0^T |b_{xx}(t)[y^{\varepsilon}(t), y^\varepsilon (t)]|^2_{\mathcal{L}_2^0} 
    +\tilde{\mathbb{E}}\big[ |b_{y\mu}(t)(\widetilde{\overline{X}}(t))[\tilde{y}^{\varepsilon}(t) , 
      \tilde{y}^\varepsilon (t)]|^2_{\mathcal{L}_2^0} \big] \\
      & \qquad\qquad  +\widetilde{\mathbb E}\big[|b_{\mu}(t)(\widetilde{\overline{X}}(t))\tilde y^{\varepsilon}(t)|^2\big]\\
    &  \qquad\qquad +\big\{ |\delta b_x (t) y^{\varepsilon}(t)|^2_{\mathcal{L}_2^0} 
    + \widetilde{\mathbb E} \big[ |\delta b_{\mu}(t)(\widetilde{\overline{X}}(t)) 
      \tilde y^{\varepsilon}(t)|^2_{\mathcal{L}_2^0} \big] \big\} \chi_{E_{\varepsilon}}(t) dt \Big)  \Big\}\\
    & = o(\varepsilon) 
\end{aligned} 
\end{equation*}
This gives \eqref{eq-4.8}.

Next, let's prove \eqref{eq-4.9}. Let $\eta^{\varepsilon}=X^{\varepsilon}(t)-\overline X(t)-y^{\varepsilon}(t)$.   
\begin{eqnarray*}
    d\eta^{\varepsilon}(t) & = &A\eta^{\varepsilon}(t)dt+\Big(a^{\theta}_x(t)\xi^{\varepsilon}(t)+\widetilde{\mathbb E}\big[ a^{\theta}_{\mu}(t)\tilde{\xi}^{\varepsilon}\big]
      -a_x(t)y^{\varepsilon}(t)\Big)dt\\
    & & +\Big( b^{\theta}_x(t)\xi^{\varepsilon}(t)+\widetilde{\mathbb E}\big[ b^{\theta}_{\mu}(t)\tilde{\xi}^{\varepsilon}\big]
     -b_x(t)y^{\varepsilon}(t)\Big)dW(t)\\ 
     & & +\delta a (t)\chi_{E_{\varepsilon}}(t)dt \\
    & = & A\eta^{\varepsilon}(t)dt + \Big( a^{\theta}_x(t)\eta^{\varepsilon}(t) + \big( a^{\theta}_x(t)-a_x(t)\big)\xi^{\varepsilon}(t) 
       + \widetilde{\mathbb E}\big[ a^{\theta}_{\mu}(t) 
        \tilde{\xi}^{\varepsilon}(t)\big] \Big)dt \\
    & & + \Big( b^{\theta}_x(t)\eta^{\varepsilon}(t)  +\big( b^{\theta}_x(t) - b_x(t)\big) \xi^{\varepsilon}(t)  
       + \widetilde{\mathbb E}\big[ b^{\theta}_{\mu}(t)  \tilde{\xi}^{\varepsilon}(t)\big]\Big)dW(t)\\
         & & +\delta a (t)\chi_{E_{\varepsilon}}(t)dt
\end{eqnarray*}
By Burkholder-Davis-Gundy inequality and Gronwall's inequality and Corollary \ref{cor3.2}, we get 
{\small 
\begin{equation}\label{eq-34}
\begin{aligned} 
\mathbb E & \sup_{0\le t\le T} |\eta^{\varepsilon}(t)|^{2} 
    \leq \cC \Big\{ \mathbb{E} \Big[ \Big( \int_0^T |a^{\theta}_x(t)-a_x(t)||\xi^\varepsilon (t)|  
       + \tilde{\mathbb{E}}\big[ |a^{\theta}_{\mu}(t)\tilde{\xi}^\varepsilon (t)| \big] dt \Big)^{2}\Big] \\
    & \qquad  +  \mathbb{E} \Big[ \Big( \int_0^T |b^{\theta}_x(t)-b_x(t)|^2_{L(H; \mathcal{L}_2^0 )} 
       |\xi^\varepsilon (t)|^2  + \tilde{\mathbb{E}}\big[ |b^{\theta}_{\mu}(t)\tilde{\xi}^\varepsilon (t)|_{\mathcal{L}_2^0}^2 \big] dt \Big) \Big]\Big\}  \\ 
       & \qquad +\mathbb{E} \Big[ \Big( \int_0^T |\delta a (t)\chi_{E_{\varepsilon}}(t)
       |dt\Big)^{2}\\
    & \leq \cC \varepsilon \Big\{\int_0^T \mathbb E \big[ |a^{\theta}_x(t)-a_x(t)|^{4} \big]^{\frac{1}{2}} 
         + \int_0^T \mathbb E \big[| b^{\theta}_x(t)-b_x(t)|^{4}\big]^{\frac 12} \Big\}+ o(\varepsilon)+ \e^2.
\end{aligned} 
\end{equation}
}
Note that by Assumption {\bf (A)}, \eqref{eq-4.11} and \eqref{eq-4.6},
\begin{equation*}
\begin{aligned}
    & \int_0^T \mathbb E \big[ |a^{\theta}_x(t)-a_x(t)|^{4} \big]^{\frac{1}{2}} dt\\
    &=\int_0^T \mathbb E \Big[ \Big|\int_0^1 \big[a_x(t,\overline X(t)+\theta\xi^{\varepsilon}(t), \mathcal L(\overline X(t)+\theta\xi^{\varepsilon}(t)) ,u^{\varepsilon}(t))\\
    & \hspace{2.9cm}-a_x(t,\overline X(t),\mathcal L(\overline X(t)),\bar u(t))\big]d\theta\Big|^{4} \Big]^{\frac{1}{2}} dt\\
    & \leq \cC \int_0^T \Big[\mathbb E\Big|\theta\xi^{\e}(t)|+\cW_2(\mathcal L(\overline X(t)+\theta\xi^{\varepsilon}(t)),\mathcal L(\overline X(t)))+\delta a_x(t)\chi_{E_{\varepsilon}}(t) \Big|^4\Big]^{\frac{1}{2}}dt\\
    & \leq \cC \Big[\int_0^T\big(\mathbb E |\xi^{\e}(t)|^4\big)^{\frac{1}{2}}dt+\e\Big]\leq \cC \e.
\end{aligned}
\end{equation*}
Similarly, 
\begin{equation*}
    \int_0^T \mathbb E \big[ |b^{\theta}_x(t)-b_x(t)|^{4} \big]^{\frac{1}{2}} dt\leq \cC \e.
\end{equation*}
Then \eqref{eq-4.9} follows from \eqref{eq-34}.
\endpf

\ss

Now we are ready to prove the following crucial estimate:

\begin{proposition}\label{prop4.4}
\begin{equation}\label{eq-4.10a}
     \mathbb E \sup\limits_{0 \leq t \leq T}|X^{\varepsilon}(t)-\overline X(t)-y^{\varepsilon}(t)-z^{\varepsilon}(t)|^2=o(\varepsilon^2).
\end{equation}
\end{proposition}

{\it Proof}. 
Let $\zeta^{\varepsilon}(t)=X^{\varepsilon}(t)-\overline X(t)-y^{\varepsilon}(t)-z^{\varepsilon}(t)$. Then, we have
\begin{equation}
\begin{cases}
    d\zeta^{\varepsilon}(t)dt=\big(A\zeta^{\varepsilon}(t)+\alpha^{\varepsilon}(t)\big)dt+\beta^{\varepsilon}(t)dW(t),\qq t\in (0,T],\\
    \zeta^{\varepsilon}(0)=0,
\end{cases}
\end{equation}
where 
\begin{equation}
\label{eq--4.20}
\begin{aligned} 
    \alpha^{\varepsilon}(t)
    &=a(t,X^{\varepsilon}(t),\mathcal L(X^{\varepsilon}(t)), u^{\varepsilon}(t))-a(t,\overline X(t),\mathcal L(\overline X(t)),\bar u(t))-a_x(t)\big[ y^{\varepsilon}(t)+z^{\varepsilon}(t)\big] \\
    &  -\widetilde{\mathbb E}\big[ a_{\mu}(t)(\widetilde{\overline{X}}(t))(\tilde y^{\varepsilon}(t)+\tilde z^{\varepsilon}(t))\big] \\ 
    & \qquad -\frac{1}{2} a_{xx}(t) [y^{\varepsilon}(t) , y^\varepsilon (t)] 
    - \frac{1}{2} \tilde{\mathbb{E}}\big[ a_{y\mu}(t)(\widetilde{\overline{X}}(t))[\tilde{y}^{\varepsilon}(t), \tilde{y}^\varepsilon (t)] \big] \\
    & -\big\{\delta a(t) + \delta a_x(t) y^{\varepsilon}(t) + \widetilde{\mathbb E}\big[\delta a_{\mu}(t)(\widetilde{\overline{X}}(t))\tilde y^{\varepsilon}(t)\big] \big\} \chi_{E_{\varepsilon}}(t),
\end{aligned} 
\end{equation}
and 
\begin{equation}
\label{eq--4.21}
\begin{aligned} 
    \beta^{\varepsilon}(t) & = b(t,X^{\varepsilon}(t),\mathcal L(X^{\varepsilon}(t)), u^{\varepsilon}(t))-b(t,\overline X(t),\mathcal L(\overline X(t)),\bar u(t))-b_x(t)\big( y^{\varepsilon}(t)+z^{\varepsilon}(t)\big)\\
    &  -\widetilde{\mathbb E}\big[ b_{\mu}(t)(\widetilde{\overline{X}}(t))(\tilde y^{\varepsilon}(t)+\tilde z^{\varepsilon}(t))\big] \\ 
    & \qquad -\frac{1}{2}b_{xx}(t)[y^{\varepsilon}(t), y^\varepsilon (t)] 
    -\frac{1}{2} \tilde{\mathbb{E}}\big[ b_{y\mu}(t)(\widetilde{\overline{X}}(t))[\tilde{y}^{\varepsilon}(t), \tilde{y}^\varepsilon (t)] \big] \\
    &  -\big\{ \delta b(t) + \delta b_x(t)y^{\varepsilon}(t) + 
    \widetilde{\mathbb E}\big[\delta b_{\mu}(t)(\widetilde{\overline{X}}(t))\tilde y^{\varepsilon}(t)\big] 
    \big\}\chi_{E_{\varepsilon}}(t).
\end{aligned}
\end{equation} 
Thus, by the Burkholder-Davis-Gundy type inequality, we have
{\small
\begin{equation}\label{eq---4.26}
\begin{aligned}
   \mathbb E \sup\limits_{0 \leq t \leq T}|\zeta^{\varepsilon}(t)|^{2}& \leq  2\mathbb E\sup\limits_{0 \leq t \leq T}\Big|\int_0^t S(t-r)\alpha^{\varepsilon}(r)dr\Big|^{2}+  2\mathbb E\sup\limits_{0 \leq t \leq T}\Big|\int_0^t S(t-r)\beta^{\varepsilon}(r)dW(r)\Big|^{2}\\
   & \leq  \cC \mathbb E\int_0^T |\alpha^{\varepsilon}(r)|^{2}dr+ \cC \mathbb E\int_0^T |\beta^{\varepsilon}(r)|_{\mathcal L_2^0}^{2}dr\\
   & \leq  \cC \mathbb E \int_0^T \Big[ |\alpha^{\varepsilon}(r)|^{2}+|\beta^{\varepsilon}(r)|_{\mathcal L_2^0}^{2} \Big]dr.
\end{aligned}
\end{equation}
}

Now, let us rewrite \eqref{eq--4.20}. Noting \eqref{eq-4.11}, we have
\begin{equation}\label{eq--4.22}
\begin{aligned}
    &  a(t,X^{\varepsilon}(t),\mathcal L(X^{\varepsilon}(t)), u^{\varepsilon}(t))-a(t,\overline X(t),\mathcal L(\overline X(t)),\bar u(t))\\
    &  -a_x(t)\big[y^{\varepsilon}(t)+z^{\varepsilon}(t)\big]-\widetilde{\mathbb E}\big[ a_{\mu}(t)(\widetilde{\overline{X}}(t))(\tilde y^{\varepsilon}(t)+\tilde z^{\varepsilon}(t))\big]-\delta a (t)\chi_{E_{\varepsilon}}(t)\\
    & =  a^{\theta}_x(t)\zeta^{\varepsilon}(t)+\widetilde{\mathbb E}\big[ a^{\theta}_{\mu}(t)\tilde{\zeta}^{\varepsilon}\big]+\big[ a^{\theta}_x(t)-a_x(t)\big]( y^{\varepsilon}(t)+ z^{\varepsilon}(t))\\
    &  +\widetilde{\mathbb E}\big[ \big( a^{\theta}_{\mu}(t)-a_{\mu}(\widetilde{\overline{X}}(t))\big)(\tilde y^{\varepsilon}(t)+\tilde z^{\varepsilon}(t))\big].
\end{aligned}
\end{equation}
Moreover, 
\begin{eqnarray}\label{eq--4.23}
   & &   a^{\theta}_x(t)-a_x(t)-\delta a_x(t)\chi_{E_{\varepsilon}}(t)\nonumber\\
   &=& \int_0^1 \int_0^1 \theta a_{xx}(t,\overline X(t)+\lambda\theta\xi^{\varepsilon}(t), \mathcal L(\overline X(t)+\lambda\theta\xi^{\varepsilon}(t)) ,u^{\varepsilon}(t))[\cdot , \zeta^{\varepsilon}(t)] d\theta d\lambda\nonumber\\ 
   & & +\int_0^1 \int_0^1 \theta \widetilde{\mathbb E}[a_{x\mu }(t,\overline X(t)+\lambda\theta\xi^{\varepsilon}(t), \mathcal L(\overline X(t)+\lambda\theta\xi^{\varepsilon}(t)) ,u^{\varepsilon}(t))[\cdot, \tilde{\zeta}^{\varepsilon}(t)]]d\theta d\lambda.
\end{eqnarray}
Again, introduce an independent probability space $(\widehat \Omega, \widehat{\mathcal F}, 
\widehat{\mathbb P})$. All processes $X$ defined on $\Omega$ will have a copy $\widehat X$ on space 
$\widehat{\Omega}$. Then 
\begin{eqnarray}\label{eq--4.24}
    & &  a^{\theta}_{\mu}(t)-a_{\mu}(t)(\widetilde{\overline{X}}(t))-\delta a_{\mu}(t)(\widetilde{\overline{X}}(t))\nonumber\\
    &=& \int_0^1\int_0^1\theta  \Big(a_{\mu x}(t,\overline X(t)+\lambda\theta\xi^{\varepsilon}(t), \mathcal L(\overline X(t)+\lambda\theta\xi^{\varepsilon}(t)) ,u^{\varepsilon}(t))[\cdot , \xi^{\varepsilon}(t)] \Big)d\theta d\lambda\nonumber\\
    & & +\int_0^1\int_0^1\theta  \Big(a_{y\mu}(t,\overline X(t)+\lambda\theta\xi^{\varepsilon}(t), \mathcal L(\overline X(t)+\lambda\theta\xi^{\varepsilon}(t)) ,u^{\varepsilon}(t))[\cdot , \tilde{\xi}^{\varepsilon}(t)] \Big)d\theta d\lambda\nonumber\\
    & & +\int_0^1\int_0^1\theta  \widehat {\mathbb E}\big[a_{\mu \mu}(t,\overline X(t)+\lambda\theta\xi^{\varepsilon}(t), \mathcal L(\overline X(t)+\lambda\theta\xi^{\varepsilon}(t)) ,u^{\varepsilon}(t))[\cdot , \widehat{\xi}^{\varepsilon}(t)]\big]d\theta d\lambda
\end{eqnarray}
Thus, from \eqref{eq--4.20}, \eqref{eq--4.22}-\eqref{eq--4.24}, and noting $\xi^{\varepsilon}(t)=\zeta^{\varepsilon}(t)+y^{\varepsilon}(t)+z^{\varepsilon}(t)$, we have 
\begin{eqnarray}\label{eq--4.25}
    \alpha ^{\varepsilon}(t)=  a^{\theta}_x(t)\zeta^{\varepsilon}(t)+\widetilde{\mathbb E}\big[ a^{\theta}_{\mu}(t)\tilde{\zeta}^{\varepsilon}\big]+\varphi^{1,\varepsilon}(t)+\varphi^{2,\varepsilon}(t)+\varphi^{3,\varepsilon}(t)+\varphi^{4,\varepsilon}(t),
\end{eqnarray}
where
{\small 
\begin{eqnarray*}
    & & \varphi^{1,\varepsilon}(t)\\
    & =& \int_0^1 \int_0^1 \theta a_{xx}(t,\overline X(t)+\lambda\theta\xi^{\varepsilon}(t), \mathcal L(\overline X(t)+\lambda\theta\xi^{\varepsilon}(t)) ,u^{\varepsilon}(t))[\zeta^{\varepsilon}(t), y^{\varepsilon}(t)+z^{\varepsilon}(t)] d\theta d\lambda\nonumber\\
   & & +\int_0^1 \int_0^1 \theta \widetilde{\mathbb E}\big[a_{x \mu }(t,\overline X(t)+\lambda\theta\xi^{\varepsilon}(t), \mathcal L(\overline X(t)+\lambda\theta\xi^{\varepsilon}(t)) ,u^{\varepsilon}(t))[\tilde{\zeta}^{\varepsilon}(t),y^{\varepsilon}(t)+z^{\varepsilon}(t)]\big]d\theta d\lambda\nonumber\\
   & & + \int_0^1\int_0^1\theta  \widetilde{\mathbb E}\big[a_{\mu x}(t,\overline X(t)+\lambda\theta\xi^{\varepsilon}(t), \mathcal L(\overline X(t)+\lambda\theta\xi^{\varepsilon}(t)) ,u^{\varepsilon}(t)[\tilde y^{\varepsilon}(t)+\tilde z^{\varepsilon}(t),\zeta^{\varepsilon}(t)] \big]d\theta d\lambda\nonumber\\
    & & +\int_0^1\int_0^1\theta   \widetilde{\mathbb E}\big[a_{y\mu}(t,\overline X(t)+\lambda\theta\xi^{\varepsilon}(t), \mathcal L(\overline X(t)+\lambda\theta\xi^{\varepsilon}(t)) ,u^{\varepsilon}(t))[\tilde y^{\varepsilon}(t)+\tilde z^{\varepsilon}(t),\tilde{\zeta}^{\varepsilon}(t)]\big]d\theta d\lambda\nonumber\\
    & & +\int_0^1\int_0^1\theta   \widehat {\mathbb E}\widetilde{\mathbb E}\big[a_{\mu \mu}(t,\overline X(t)+\lambda\theta\xi^{\varepsilon}(t), \mathcal L(\overline X(t)+\lambda\theta\xi^{\varepsilon}(t)) ,u^{\varepsilon}(t))[\tilde{\zeta}^{\varepsilon}(t),\widehat y^{\varepsilon}(t)+\widehat z^{\varepsilon}(t)]\big]d\theta d\lambda, 
\end{eqnarray*} 

 \begin{eqnarray*}    
     & & \varphi^{2,\varepsilon}(t)\\
     &=& \int_0^1 \int_0^1 \theta a_{xx}(t,\overline X(t)+\lambda\theta\xi^{\varepsilon}(t), \mathcal L(\overline X(t)+\lambda\theta\xi^{\varepsilon}(t)) ,u^{\varepsilon}(t))[y^{\varepsilon}(t)+z^{\varepsilon}(t), y^{\varepsilon}(t)+z^{\varepsilon}(t)] \\
    & & \qquad\qquad   -\theta a_{xx}(t,\overline X(t)+\lambda\theta\xi^{\varepsilon}(t), 
    \mathcal L(\overline X(t)+\lambda\theta\xi^{\varepsilon}(t)) ,u^{\varepsilon}(t))[y^{\varepsilon}(t), 
    y^{\varepsilon}(t)] d\theta d\lambda\nonumber\\
   &  &+  2\int_0^1\int_0^1\theta  \widetilde{\mathbb E}\big[a_{x\mu }(t,\overline X(t)+\lambda\theta\xi^{\varepsilon}(t), \mathcal L(\overline X(t)+\lambda\theta\xi^{\varepsilon}(t)) ,u^{\varepsilon}(t)[\tilde y^{\varepsilon}(t)+\tilde z^{\varepsilon}(t), \\
   & & \hspace{10cm}y^{\varepsilon}(t)+z^{\varepsilon}(t)]\big]d\theta d\lambda\nonumber\\
    &  &+ \int_0^1\int_0^1\theta   \widetilde{\mathbb E}\big[a_{y\mu}(t,\overline X(t)+
    \lambda\theta\xi^{\varepsilon}(t), \mathcal L(\overline X(t)+\lambda\theta\xi^{\varepsilon}(t)),  
    u^{\varepsilon}(t))[\tilde y^{\varepsilon}(t)+\tilde z^{\varepsilon}(t),\\ 
   & & \hspace{10cm} \tilde y^{\varepsilon}(t)+\tilde z^{\varepsilon}(t)]\big]d\theta d\lambda\nonumber\\
    &  &+ \int_0^1\int_0^1\theta   \widehat {\mathbb E}\widetilde{\mathbb E}\big[a_{\mu \mu}(t,\overline X(t)+\lambda\theta\xi^{\varepsilon}(t), \mathcal L(\overline X(t)+\lambda\theta\xi^{\varepsilon}(t)) ,u^{\varepsilon}(t))[\tilde y^{\varepsilon}(t)+\tilde z^{\varepsilon}(t),\\
    & & \hspace{10cm}\widehat y^{\varepsilon}(t)+\widehat z^{\varepsilon}(t)]\big]d\theta d\lambda,\\
\end{eqnarray*}
\begin{eqnarray*}
    & & \varphi^{3,\varepsilon}(t)\\
    &=& \int_0^1 \int_0^1 \theta \Big(a_{xx}(t,\overline X(t)+\lambda\theta\xi^{\varepsilon}(t), \mathcal L(\overline X(t)+\lambda\theta\xi^{\varepsilon}(t)) ,u^{\varepsilon}(t))-a_{xx}(t)\Big)[y^{\varepsilon}(t), y^\varepsilon (t)] d\theta d\lambda,\\
\end{eqnarray*}   
\begin{eqnarray*}
    \varphi^{4,\varepsilon}(t)&=& \big\{\delta a_x(t)z^{\varepsilon}(t)+\widetilde {\mathbb E}[\delta a_{\mu}(t)(\widetilde{\overline{X}}(t))\tilde z^{\varepsilon}(t)]\big\}\chi_{E_{\varepsilon}}(t).
\end{eqnarray*}
}

Now we estimate each $\varphi^{i,\varepsilon},\ i=1,\ 2,\ 3,\ 4 $. By the estimates in Proposition \ref{prop4.3} and Proposition \ref{prop-3.2}, we have 
\begin{equation}\label{eq--4.26}
   \mathbb E\big[\sup_{t\in [0,T]}|\varphi^{1,\varepsilon}(t)|^{2}\big]\leq \cC \varepsilon^{3},
\end{equation} 
With Corollary \ref{cor3.1}, and the boundedness of $a_{x\mu}$, 
\begin{equation}\label{eq--4.27}
\begin{aligned}
   &  \mathbb E\Big[ \int_0^T \big| \widetilde {\mathbb E}\big[a_{x\mu }(t,\overline X(t)+\lambda\theta\xi^{\varepsilon}(t), \mathcal L(\overline X(t)+\lambda\theta\xi^{\varepsilon}(t)) ,u^{\varepsilon}(t))[\tilde y^{\varepsilon}(t), y^{\varepsilon}(t)]\big] \big| dt \Big]\\
    &\leq  \cC \int_0^T \mathbb E \big|\widetilde {\mathbb E} 
    \big[ a_{x\mu } (t,\overline X(t)+\lambda\theta\xi^{\varepsilon}(t), \mathcal L(\overline X(t)+\lambda\theta\xi^{\varepsilon}(t)) ,u^{\varepsilon}(t))[\cdot , \tilde y^{\varepsilon}(t)]\big] 
    \big|^2 dt \Big)^{\frac{1}{2}}\\
    & \times \Big(\mathbb E \sup_{t\in [0,T]}|y^{\varepsilon}(t)|^2\Big)^{\frac{1}{2}}\\
    & \leq  \cC o(\varepsilon),
\end{aligned}
\end{equation}

thus, 
{\small
\begin{equation}\label{eq--4.28}
\begin{aligned}
   &\mathbb E\Big[ \int_0^T \big| \widetilde{\mathbb E}\big[a_{x\mu }(t,\overline X(t)+\lambda\theta\xi^{\varepsilon}(t), \mathcal L(\overline X(t)+\lambda\theta\xi^{\varepsilon}(t)) ,u^{\varepsilon}(t))[\tilde y^{\varepsilon}(t)+\tilde z^{\varepsilon}(t),y^{\varepsilon}(t)+z^{\varepsilon}(t)] \big]\big| dt\Big]\\
   &\leq \cC o(\varepsilon).
\end{aligned}
\end{equation}
}
Similarly, the terms $a_{\mu \mu}$, $a_{y\mu}$ of $\varphi^{2,\varepsilon}(t)$ satisfy
\begin{equation}\label{eq--4.29}
\begin{cases}
    \mathbb E\Big[\int_0^T \big|\widetilde{\mathbb E}\big[a_{y\mu}(t,\overline X(t)+\lambda\theta\xi^{\varepsilon}(t), \mathcal L(\overline X(t)+\lambda\theta\xi^{\varepsilon}(t)) ,u^{\varepsilon}(t))\\
    \hspace{6cm}\times [\tilde y^{\varepsilon}(t)+\tilde z^{\varepsilon}(t), \tilde y^{\varepsilon}(t)+\tilde z^{\varepsilon}(t)]\big]\big| dt\Big]\leq \cC o(\varepsilon),\\
    \mathbb E\Big[\int_0^T \big|\widehat {\mathbb E} \widetilde{\mathbb E}\big[a_{\mu\mu }(t,\overline X(t)+\lambda\theta\xi^{\varepsilon}(t), \mathcal L(\overline X(t)+\lambda\theta\xi^{\varepsilon}(t)) ,u^{\varepsilon}(t))\\
     \hspace{6cm}\times [\tilde y^{\varepsilon}(t)+\tilde z^{\varepsilon}(t), \widehat y^{\varepsilon}(t)+\widehat z^{\varepsilon}(t)] \big]\big| dt\Big]\leq \cC o(\varepsilon).
\end{cases}
\end{equation}
Combining \eqref{eq--4.28} and \eqref{eq--4.29}, we get 
\begin{equation}\label{eq--4.30}
    \mathbb E\big[\sup_{t\in [0,T]}|\varphi^{2,\varepsilon}(t)|^{2}\big]\leq \cC o(\varepsilon^2).
\end{equation}

Furthermore, 
\begin{eqnarray}\label{eq--4.31}
    & & \mathbb E\Big(\int_0^T |\varphi^{3,\varepsilon}(t)|dt\Big)\\
    &\leq & \cC \mathbb E\Big( \sup_{t\in [0,T]}|y^{\varepsilon}(t)|^2\int_0^T\int_0^1\int_0^1 \big|a_{xx}(t,\overline X(t)+\lambda\theta\xi^{\varepsilon}(t), \mathcal L(\overline X(t)+\lambda\theta\xi^{\varepsilon}(t)) ,u^{\varepsilon}(t))\nonumber\\
    & & \hspace{6cm} -a_{xx}(t) \big| d \theta d\lambda dt   \Big)\nonumber\\
    &\leq & \cC \mathbb E\Big( \sup_{t\in [0,T]}|y^{\varepsilon}(t)|^4\Big)^{\frac{1}{2}}\Big(\mathbb E\big( \sup_{t\in [0,T]}|X^{\varepsilon}(t)-\overline X(t)|^2\big)^{\frac{1}{2}}+|E_{\varepsilon}|\Big)\nonumber\\
    &\leq & \cC \varepsilon(\varepsilon^{\frac{1}{2}}+\varepsilon). \nonumber
\end{eqnarray}
Similarly,
\begin{eqnarray}\label{eq--4.32}
    & & \mathbb E\Big(\int_0^T |\varphi^{4,\varepsilon}(t)|^{2}dt\Big)\leq  \cC \mathbb E\Big( \sup_{t\in [0,T]}|z^{\varepsilon}(t)|^{2}\Big)| E_{\varepsilon}|\nonumber\leq \cC o(\varepsilon^{2}).
\end{eqnarray}
Following the same process, we get the corresponding estimates for $\beta ^{\varepsilon}(t)$. Combining \eqref{eq---4.26}, \eqref{eq--4.25}, \eqref{eq--4.26}, \eqref{eq--4.30}-\eqref{eq--4.32}, and by Gronwall's inequality, 
\begin{equation*}
    \mathbb E\Big(\sup_{t\in [0,T]}| \zeta ^{\varepsilon}(t)|^{2}\Big) =  o(\varepsilon^{2}). 
\end{equation*}
\endpf


\section{Well-Posedness within the Transposition Solution Framework}\label{sec-trans}

In this section, we establish the well-posedness of the adjoint equations arising in our maximum principle. And for the convience of readers, we recall the notion of a relaxed transposition solution for the second-order adjoint equation, which is an $\mathcal L(H)-$valued BSEE (eg. \cite{Lu2014}, \cite{Lu2021}). 

\subsection{Transposition Solution to an $H$-valued McKean-Vlasov BSEE}\label{subsec-trans}
In this subsection, we establish the well-posedness of an $H$-valued McKean–Vlasov BSEE in the sense of a so-called transposition solution, which will be defined below. 

The system under consideration is given as follows.
\begin{equation}\label{system1}
\begin{cases}
    dY(t)dt=-A^*Y(t)dt+F(t,Y(t),Z(t),\mathcal L(Y(t),Z(t)))dt+Z(t)dW(t),\ t\in [0,T),\\
    Y(T)=Y_T.
\end{cases}
\end{equation}
We make the following assumptions for the system
\eqref{system1} 

\ss

{\bf (B1)} {\it  Suppose that $Y_T\in L^2_{\mathcal F_T}(\Omega;H)$, $F(\cdot, \cdot,\cdot,\cdot): [0,T]\times H\times \mathcal L_2^0\times  \mathcal P_2(H) \to H$
    is measurable in the time variable and continuous with respect to the rest of the variables.}

\ss
{\bf (B2)} {\it
For any $(t,y_1,y_2,z_1,z_2,\mu_1,\mu_2)\in [0,T]\times H \times H\times \mathcal L_2^0\times \mathcal L_2^0\times\mathcal P_2(H)\times \mathcal P_2(H) $, and $\mu_1,\ \mu_2$ having the same first marginal on $H$,

$$|F(t,y_1,z_1,\mu_1)-F(t,y_2,z_2,\mu_2)|_H\leq \cC\Big(|y_1-y_2|_H+|z_1-z_2|_{\mathcal L_2^0}+ \cW_2(\mu_1,\mu_2)\Big)$$
 }

\ss

Following the idea of the standard transposition solution in \cite{Lu2015} and 
\cite[Chapter 4, Section 4.3]{Lu2021}, in order to define the transposition solution to 
system \ref{system1}, we introduce the following (forward) stochastic evolution equation: 
\begin{equation}\label{system2}
\begin{cases}
    dx(r)=Ax(r)dr+a(t,x(r),\mathcal L(x(r)))dr+b(r,x(r),\mathcal L(x(r)))dW(r),\ r\in (t,T]\\
    x(t)=\eta \in L^2_{\mathcal F_t}(\Omega;H).
\end{cases}    
\end{equation}
where $t\in [0,T]$, $a, \ b$ are measurable in the first variable, continuous with 
respect to the rest of the variables, and for any $(t,x_1,x_2,\nu_1,\nu_2,\mu_1,\mu_2) 
\in [0,T]\times H\times H \times \mathcal P_2(H)\times \cP_2(H)\times \cP_2(H) 
\times \cP_2(H)$,
$$|a(t,x_1,\nu_1)-a(t,x_2,\nu_2)|_H\leq C\Big(|x_1-x_2|_H+\cW_2(\nu_1,\nu_2)\Big),$$
$$|b(t,x_1,\mu_1)-a(t,x_2,\mu_2)|_{\cL_2^0}\leq C\Big(|x_1-x_2|_H+\cW_2(\mu_1,\mu_2)\Big).$$
Then by \cite[Theorem 4.2]{Ahmed2015}, system \ref{system2} has a unique mild solution:
{\small
$$
X(t) = S(t)\zeta + \int_0^t S(t-\tau)a(\tau,X(\tau),\cL(X(\tau)))d\tau 
+ \int_0^t S(t-\tau)b(\tau,X(\tau),\cL(X(\tau)))dW(\tau), 
$$
}

We now introduce the following notion.

\begin{definition}\label{def1}
    We call $(Y(\cdot),Z(\cdot))\in D_{\dbF}([0,T];L^2(\Omega;H))\times L^2_\dbF(0,T;L^2(\Om;\cL_2^0))$ a transposition solution to system \eqref{system1} if for any $t\in [0,T]$, $X(\cdot)\in C_{\dbF}([t,T];L^2(\Omega;H))$ being the mild solution to system \eqref{system2}, it holds that $\langle X(\cdot),F(\cdot,Y(\cdot),Z(\cdot),\cL(Y(\cdot)Z(\cdot)))\rangle_H\in L_{\mathbb F}^1(\Om;L^1(t,T))$, and 
    \begin{eqnarray*}
        & & \mathbb E\langle X(T),Y_T\rangle_H-\mathbb E\int_t^T \langle X(s),F(s,Y(s),Z(s),\cL(Y(s),Z(s)))\rangle_Hds\\
        &=& \mathbb E\langle X(t),Y(t)\rangle_H+ \mathbb E\int_t^T \langle a(s,X(s),\cL(X(s))),Y(s)\rangle_Hds\\
        & & \hspace{2.5cm}+\mathbb E\int_t^T \langle b(s,X(s),\cL(X(s))),Z(s)\rangle_{\cL_2^0}ds.
    \end{eqnarray*}    
\end{definition}

Now, let us show the well-posedness of system \eqref{system1} in the sense of transposition solution. 

\begin{theorem}\label{thm1}
    Under Assumptions \textbf{(B1)-(B2)}, let $F(\cdot,0,0,0)\in L_{\mathbb F}^1(0,T;L^2(\Om;H))$. Then system \eqref{system1} admits a unique transposition solution $D_{\dbF}([0,T];L^2(\Omega;H))$\\
    $\times L^2_\dbF(0,T;L^2(\Om;\cL_2^0))$. Furthermore,
    \begin{equation*}
    \begin{aligned}
    &| Y(\cdot), Z(\cdot))|_{D_{\dbF}([0,T];L^2(\Omega;H))\times L^2_\dbF(0,T;L^2(\Om;\cL_2^0))}\\ 
    &\leq \cC \Big(|Y_T|_{L_{\cF_T}^2(\Omega;H)}+|F(\cdot,0,0,0)|_{L_{\mathbb F}^1(0,T;L^2(\Om;H))}\Big).            
    \end{aligned}
    \end{equation*}
\end{theorem}

{\it Proof.} Fix any $T_1\in [0,T)$, for any $(Y^{\prime}(\cdot),Z^{\prime}(\cdot))\in D_{\dbF}([T_1,T];L^2(\Omega;H))\times L^2_\dbF(T_1,T;$\\
$L^2(\Om;\cL_2^0))$, $F(\cdot,Y^{\prime}(\cdot),Z^{\prime}(\cdot), \mathcal L(Y^{\prime}(\cdot),Z^{\prime}(\cdot)))\in  L_{\mathbb F}^1(T_1,T;L^2(\Om;H))$.
Consider the following equation 
{\small
\begin{equation}\label{system3}
\begin{cases}
    dY(t)dt=-A^*Y(t)dt+F(t,Y^{\prime}(\cdot),Z^{\prime}(\cdot),\mathcal L(Y^{\prime}(\cdot),Z^{\prime}(\cdot)))dt+Z(t)dW(t),\ t\in [T_1,T),\\
    Y(T)=Y_T.
\end{cases}
\end{equation}
}
is a standard non McKean-Vlasov BSEE. 
Then by \cite[Chapter 4]{Lu2021}, it admits a unique transposition solution 
$(Y(\cdot),Z(\cdot))\in D_{\dbF}([T_1,T];L^2(\Omega;H))\times L^2_\dbF(T_1,T;L^2(\Om;\cL_2^0))$. This defines a map $\Phi$ from $D_{\dbF}([T_1,T];L^2(\Omega;H))\times L^2_\dbF(T_1,T;\\
L^2(\Om;\cL_2^0))$ into itself by $(Y^{\prime },Z^{\prime})\mapsto (Y,Z)$. 

We claim that for $T_1$ being sufficiently close to $T$, $\Phi$ is a contraction, i.e., 
\begin{eqnarray}\label{eq-3.4}
    & & \Big|\Phi(Y^{\prime },Z^{\prime})-\Phi(Y^{\prime\prime },Z^{\prime\prime})\Big|_{D_{\dbF}([T_1,T];L^2(\Omega;H))\times L^2_\dbF(T_1,T;L^2(\Om;\cL_2^0))}\\
    &\leq & \frac{1}{2}\big|(Y^{\prime },Z^{\prime})-(Y^{\prime\prime },Z^{\prime\prime})\big|_{D_{\dbF}([T_1,T];L^2(\Omega;H))\times L^2_\dbF(T_1,T;L^2(\Om;\cL_2^0))},\nonumber\\
    & & \qq \forall \   (Y^{\prime },Z^{\prime}),\ (Y^{\prime\prime },Z^{\prime\prime})\in D_{\dbF}([T_1,T];L^2(\Omega;H))\times L^2_\dbF(T_1,T;L^2(\Om;\cL_2^0)).\nonumber
\end{eqnarray}
To show this, put $\hat f(\cdot)=F(\cdot,Y^{\prime },Z^{\prime},\mathcal L(Y^{\prime }(\cdot),Z^{\prime}(\cdot)))-F(\cdot,Y^{\prime\prime }(\cdot),Z^{\prime\prime }(\cdot),\mathcal L(Y^{\prime\prime }(\cdot),Z^{\prime\prime }(\cdot)))$ and $(\hat Y(\cdot),\hat Z(\cdot))=\Phi(Y^{\prime },Z^{\prime})-\Phi(Y^{\prime\prime },Z^{\prime\prime}) $. Then $(\hat Y(\cdot), \hat Z(\cdot))$ satisfies the following standard non McKean-Vlasov BSEE
\begin{equation*}
    \begin{cases}
        d\hat Ydt=-A\hat Y(t)dt+\hat f(t)dt+\hat Z(t)dW(t), \ t\in [T_1,T),\\
        \hat Y(T)=0.
    \end{cases}
\end{equation*}
Applying the well-posedness result from \cite[Chapter 4]{Lu2021}, together with assumption {\bf (B2)}, we obtain the following estimate:
{\small
\begin{equation}\label{eq-3.5}
\begin{aligned}
   &   \qquad |(\hat Y(\cdot), \hat Z(\cdot))|_{D_{\dbF}([T_1,T];L^2(\Omega;H))\times L^2_\dbF(T_1,T;L^2(\Om;\cL_2^0))}\\
   &  \leq  \cC |F(\cdot,Y^{\prime}(\cdot),Z^{\prime}(\cdot),\mathcal L(Y^{\prime}(\cdot),Z^{\prime}(\cdot))-F(\cdot,Y^{\prime\prime }(\cdot),Z^{\prime\prime }(\cdot),\mathcal L(Y^{\prime\prime }(\cdot),Z^{\prime\prime }(\cdot)))|_{L_{\mathbb F}^1(T_1,T;L^2(\Om;H))}\\
   & \leq  \cC\int_{T_1}^T\Big( \mathbb E\Big(|Y^{\prime}(t)-Y^{\prime\prime }(t)|_H+|Z^{\prime}(t)-Z^{\prime\prime }(t)|_{\cL_2^0}\\
   &\hspace{2.1cm}+\cW_2(\cL(Y^{\prime}(t),Z^{\prime}(t)),\cL(Y^{\prime\prime }(t),Z^{\prime\prime }(t)))\Big)^2\Big)^{\frac{1}{2}}dt\\
   & \leq  \cC \int_{T_1}^T\Big( \mathbb E|Y^{\prime}(t)-Y^{\prime\prime }(t)|_H^2\Big)^{\frac{1}{2}}dt+\int_{T_1}^T\Big( \mathbb E|Z^{\prime}(t)-Z^{\prime\prime }(t)|_{\cL_2^0}^2\Big)^{\frac{1}{2}}dt   \\
   &  +\int_{T_1}^T  \Big(\mathbb E\Big(\cW_2(\cL(Y^{\prime}(t),Z^{\prime}(t)),\cL(Y^{\prime\prime }(t),Z^{\prime\prime }(t)))\Big)^2\Big)^{\frac{1}{2}}dt\\
   & \leq  \cC(T-T_1) \Big(\sup_{t\in [T_1,T]} \mathbb E|Y^{\prime}(t)-Y^{\prime\prime }(t)|_H^2\Big)^{\frac{1}{2}}+ \cC (T-T_1)^{\frac{1}{2}}\Big(\int_{T_1}^T\Big( \mathbb E|Z^{\prime}(t)-Z^{\prime\prime }(t)|_{\cL_2^0}^2\Big)dt\Big)^{\frac{1}{2}}\\
   &  + \cC \int_{T_1}^T\Big( \mathbb E|Y^{\prime}(t)-Y^{\prime\prime }(t)|_H^2+\mathbb E|Z^{\prime}(t)-Z^{\prime\prime }(t)|_{\cL_2^0}^2\Big)^{\frac{1}{2}}dt\\
   & \leq  \cC (T-T_1)^{\frac{1}{2}} |( Y^{\prime}(\cdot)-Y^{\prime\prime }(\cdot), Z^{\prime}(\cdot)-Z^{\prime\prime }(\cdot))|_{D_{\dbF}([T_1,T];L^2(\Omega;H))\times L^2_\dbF(T_1,T;L^2(\Om;\cL_2^0))}.
\end{aligned}
\end{equation}
}
Let us choose $T_1$, such that 
\begin{equation}\label{eq-1.5}
    \cC (T-T_1)^{1/2}\leq \frac{1}{2}.
\end{equation}
Then $\Phi$ is contractive. Hence it admits a unique fixed point, which is a transposition solution to \eqref{system3} on $[T_1,T]$.

Moreover, when $T_1$ satisfies \eqref{eq-1.5}, an argument similar to \eqref{eq-3.5} yields 
\begin{eqnarray}\label{eq-1.6}
   & &  | Y(\cdot), Z(\cdot))|_{D_{\dbF}([T_1,T];L^2(\Omega;H))\times L^2_\dbF(T_1,T;L^2(\Om;\cL_2^0))}\\
   &  \leq & \cC \Big(|Y_T|_{L_{\cF_T}^2(\Omega;H)}+|F(\cdot,Y^{\prime}(\cdot),Z^{\prime}(\cdot),\mathcal L(Y^{\prime}(\cdot),Z^{\prime}(\cdot)))|_{L_{\mathbb F}^1(T_1,T;L^2(\Om;H))}\Big)\nonumber\\
   & \leq & \cC \Big(|Y_T|_{L_{\cF_T}^2(\Omega;H)}+|F(\cdot,0,0,0)|_{L_{\mathbb F}^1(T_1,T;L^2(\Om;H))}\Big)\nonumber\\
   & & +\cC \Big( (T-T_1)+(T-T_1)^{1/2}\Big) |( Y^{\prime}(\cdot), Z^{\prime}(\cdot))|_{D_{\dbF}([T_1,T];L^2(\Omega;H))\times L^2_\dbF(T_1,T;L^2(\Om;\cL_2^0))}\nonumber\\
   & \leq & \cC \Big(|Y_T|_{L_{\cF_T}^2(\Omega;H)}+|F(\cdot,0,0,0)|_{L_{\mathbb F}^1(T_1,T;L^2(\Om;H))}\Big).\nonumber
\end{eqnarray}
Repeating the above argument, we obtain the existence of transposition solution to \eqref{system1}. The uniqueness follows from estimate \eqref{eq-1.6}. 
\endpf



\subsection{Regularity of second-order adjoint equation}

In this section, we give a brief introduction of the relaxed transposition solution, which gives the well-posedness to \eqref{ad-eq2}, see \cite{Lu2014}, \cite{Lu2021} for more detailed discussion.

Consider the following formally $\mathcal L(H)$-valued backward stochastic evolution equation:
\begin{equation}\label{eq-4.1-3}
\begin{cases}
    dP=-(A^*+J^*)Pdt-P(A+J)dt-K^*PKdt\\
    \hspace{1cm} -(K^*Q+QK)dt+Fdt+QdW(t),\hspace{1cm} \mbox{in } [0,T),\\
    P(T)=P_T,
\end{cases}
\end{equation}
where (q=p/(p-1)),
\begin{equation}\label{eq-4.1-4}
    J\in L_{\mathbb F}^{2q}(0,T;L^{\infty}(\Omega;\mathcal L(H))),\ K\in L_{\mathbb F}^{2q}(0,T;L^{\infty}(\Omega;\mathcal L(H;\mathcal L_2^0))),
\end{equation}
\begin{equation}\label{eq-4.1-5}
    F\in L_{\mathbb F}^1(0,T;L^{p}(\Omega;\mathcal L(H))),\ P_T\in L_{\mathcal F_T}^{p}(\Omega;\mathcal L(H)).
\end{equation}

Let
\begin{equation*}
    \begin{array}{ll}\ds
        \mathcal{L}_{pd}(L_{\mathbb{F}}^2(0,T;L^{4}(\Omega,H));L^2_{\mathbb{F}}(0,T;L^{\frac{4}{3}}(\Omega,H)))\\
        \ns\ds \deq \Big\{L\!\in\! \cL\big(L_{\mathbb{F}}^2(0,T;L^{4}(\Omega,H));L^2_{\mathbb{F}}(0,T;L^{\frac{4}{3}}(\Omega,H))\big) \big| \mbox{for }\ae (t,\omega)\in [0,T]\times\Omega, \mbox{there
            is }\\
        \ns\ds\q \wt L(t,\omega)\!\in\!\mathcal{L}(H)\;  \mbox{such that } \big( L v(\cd)\big)(t,\omega)
        = \wt L (t,\omega)v(t,\omega),
        \forall\; v(\cd)\in
        L_{\mathbb{F}}^2(0,T;L^{4}(\Omega,H))\Big\}.
    \end{array}
\end{equation*}
In the sequel, if there is no confusion,  we identify $L\in \mathcal{L}_{pd}(L_{\mathbb{F}}^2(0,T;L^{4}(\Omega,H));$\\
$L^2_{\mathbb{F}}(0,T; L^{\frac{4}{3}}(\Omega,H)))$ with $\wt L(\cd,\cd)$.

Let
\begin{eqnarray*}
    \cP[0,T] & \deq & \big\{P(\cdot,\cdot)\ |\ P(\cdot,\cdot)\in \mathcal{L}_{pd}(L_{\mathbb{F}}^2(0,T;L^{4}(\Omega,H));L^2_{\mathbb{F}}(0,T;L^{\frac{4}{3}}(\Omega,H))),\\
    & & \ \ \ P(\cdot,\cdot)\xi \in D_{\mathbb{F}}([t,T];L^{\frac{4}{3}}(\Omega,H))) \ \textup{and} \ |P(\cdot,\cdot)\xi|_{D_{\mathbb{F}}([t,T];L^{\frac{4}{3}}(\Omega,H))}\\
    & & \ \ \ \leq C |\xi|_{L_{\mathcal{F}_t}^{4}(\Omega;H)} \ \textup{for every}\ t\in [0,T] \ \textup{and} \ \xi \in L_{\mathcal{F}_t}^{4}(\Omega;H)\big\},
\end{eqnarray*}
and
\begin{eqnarray*}
 \mathcal{Q} [0,T]
& \deq &\big\{(Q^{(\cdot)},\widehat{Q}^{(\cdot)})\ |\ Q^{(t)},\widehat{Q}^{(t)}\in \mathcal{L}(\mathcal{H}_t;L_{\mathbb{F}}^2(t,T;L^{\frac{4}{3}}(\Omega;\mathcal{L}_2^0)))\\
    & & \hspace{0.8cm} \textup{and}\  Q^{(t)}(0,0,\cdot)^{*}=\widehat{Q}^{(t)}(0,0,\cdot) \ \textup{for any}\  t\in [0,T)\big\}
\end{eqnarray*}
with
\begin{equation*}
    \mathcal{H}_t\deq L_{\mathcal{F}_t}^{4}(\Omega;H)\times L_{\mathbb{F}}^2(t,T;L^{4}(\Omega;H))\times L_{\mathbb{F}}^2(t,T;L^{4}(\Omega;\mathcal{L}_2^0)),\ \ \forall \  t\in [0,T).
\end{equation*}
For $j=1, 2$ and $t\in [0,T)$, consider the following equation:
\begin{equation}\label{test-eq1}
    \begin{cases}
        d\f_j=(A+J)\f_jds+u_jds+K\f_jdW(s)+v_jdW(s) &\textup{ in } (t,T],\\
        \f_j(t)=\xi_j
    \end{cases}
\end{equation}
where $\xi_j \in L_{\mathcal{F}_t}^{4}(\Omega;H)$, $u_j\in L_{\mathbb{F}}^2(t,T;L^{4}(\Omega;H)) $ and $v_j \in
L_{\mathbb{F}}^2(t,T; L^{4}(\Omega;\mathcal{L}_2^0))$.  By the classical well-posedness result for SEEs, we know that \eqref{test-eq1} has a unique mild solution $\f_j\in C_\dbF([t,T];L^4(\Om;H))$ (e.g.,\cite[Section 3.2]{Lu2021}).

\begin{definition}\label{def2.1} A 3-tuple $(P(\cdot), Q^{(\cdot)},\widehat{Q}^{(\cd)})\in
    \mathcal{P}[0,T]\times \mathcal{Q}[0,T]$ is called a
    relaxed transposition solution to the equation \eqref{eq-4.1-3}  if for any $t\in
    [0,T]$, $\xi_j \in L_{\mathcal{F}_t}^{4}(\Omega; H)$, $
    u_j(\cdot)\in L_{\mathbb{F}}^2(t,T;L^{4}(\Omega;H))$
    and $v_j(\cdot)\in
    L_{\mathbb{F}}^2(t,T;L^{4}(\Omega;\mathcal{L}_2^0))$ ($j=1,2$), it holds that
    \begin{eqnarray*}
        & & \mathbb{E}\lan P_T\f_1(T),\f_2(T)\ran_H-\mathbb{E}\int_t^T \lan F(s)\f_1(s),\f_2(s)\ran_H ds\\
        & & =\mathbb{E}\langle P(t)\xi_1,\xi_2\rangle_H+\mathbb{E}\int_t^T \lan P(s)u_1(s),\f_2(s)\ran_Hds+\mathbb{E}\int_t^T \lan P(s)\f_1(s),u_2(s)\ran_Hds\\
        & & \ \ \ +\mathbb{E}\int_t^T \lan P(s)K(s)\f_1(s),v_2(s)\ran_{\mathcal{L}_2^0}ds +\mathbb{E}\int_t^T \lan P(s)v_1(s),K(s)\f_2(s)+v_2(s)\ran_{\mathcal{L}_2^0}ds\\
        & & \ \ \ +\mathbb{E}\int_t^T \lan v_1(s),\widehat{Q}^{(t)}(\xi_2,u_2,v_2)(s)\ran_{\mathcal{L}_2^0}ds +\mathbb{E}\int_t^T \lan  Q^{(t)}(\xi_1,u_1,v_1)(s),v_2(s)\ran_{\mathcal{L}_2^0}ds
    \end{eqnarray*}
\end{definition}

As an immediate corollary of \cite[Theorem 12.9]{Lu2021}, we have
the following well-posedness result for the equation \eqref{ad-eq2}.

\begin{proposition}\label{prp4.1} For any $F,\ J, \ K$ and $P_T$ satisfying \eqref{eq-4.1-4} and \eqref{eq-4.1-5}, the equation \eqref{eq-4.1-3}
    admits a unique relaxed transposition solution $(P(\cdot),
    Q^{(\cdot)}, \widehat{Q}^{(\cdot)})$. Furthermore,
    \begin{eqnarray}
        & & |P|_{\mathcal{L}(L_{\mathbb{F}}^2(0,T;L^{4}(\Omega,H));L^2_{\mathbb{F}}(0,T;L^{\frac{4}{3}}(\Omega,H)))} + \sup\limits_{t\in[0,T)}|(Q^{(t)},\widehat{Q}^{(t)})|_{ \mathcal{L}(\mathcal{H}_t;L_{\mathbb{F}}^2(t,T;L^{\frac{4}{3}}(\Omega;\mathcal{L}_2^0)))^2}\nonumber\\
        & & \leq C\big(|F|_{L_{\mathbb{F}}^1(0,T;L^2(\Omega;\mathcal{L}(H)))}+|P_T|_{L_{\mathcal{F}_T}^2(\Omega;\mathcal{L}(H))}\big).
    \end{eqnarray}
\end{proposition}

Now we verify that the coefficients of equation \eqref{ad-eq2} satisfy \eqref{eq-4.1-4} and \eqref{eq-4.1-5}. In this case, we have that 
$$
\begin{cases}
J(t)=a_x(t),\q K(t)=b_x(t), \\
\ns\ds F(t) = \mathbb{H}_{xx}(t)  
+ \widetilde{\mathbb E} \big[\widetilde{\mathbb H}_{y\mu}(t)(\overline{X}(t))\big] ,\\
\ns\ds P_T = - h_{xx}\left(\overline{X}(T), \mathcal L\left( \overline{X}(T)\right)\right) 
   -\tilde{\mathbb{E}}\big[h_{y\mu} (\widetilde{\overline{X}}(T), \mathcal L(\widetilde{\overline{X}}(T)))(\overline{X}(T))\big].
\end{cases}
$$
From Assumption {\bf (A)}, it is clear that $J,\ K, F, \ P_T$ satisfy \eqref{eq-4.1-4} and \eqref{eq-4.1-5}. Therefore, we have the well-posedness of the second-order adjoint equation \eqref{ad-eq2} in relaxed tranposition solution sense:
\begin{proposition}\label{prop-4.2}
Under Assumption {\bf (A)}, equation \eqref{ad-eq2} admits a unique relaxed transposition solution.    
\end{proposition}
%

%


    %
    %



\section{Proof of Pontryagin Maximum Principle}

With the estimates we established in Subsection \ref{subsec-VE}, we are now ready to prove our main result Theorem \ref{thm4.1} of this paper.  

Let us compute the variation of the cost functional. Using Proposition \ref{prop4.3}, \ref{prop-3.2} and \ref{prop4.4} 
\begin{equation}\label{eq-3.11}
\begin{aligned}
    0 & \leq   \cJ(u^{\e}(\cdot))-\cJ(\bar u(\cdot))\\
    & =  \mathbb E\Big[\int_0^T \big\{f(t, X^{\varepsilon}(t),\mathcal L( X^{\varepsilon}(t)),u^{\varepsilon}(t))- f(t,\overline X(t),\mathcal L(\overline X(t)),\bar u(t))\big\}dt\Big]\\
    &  +\mathbb E\Big[h(X^{\varepsilon}(T),\mathcal L(X^{\varepsilon}(T))-h(\overline X(T),\mathcal L (\overline X(T)))\Big]\\
    & = \mathbb E\Big[\int_0^T \big\{f_x(t)(y^{\varepsilon}(t)+z^{\varepsilon}(t))+\widetilde{\mathbb E}\big [f_{\mu}(t)(\widetilde{\overline{X}}(t))(\tilde y^{\varepsilon}(t)+\tilde z^{\varepsilon}(t))\big]\big\}dt\Big] \\
        &  + \mathbb E \Big[h_x(\overline X(T),\mathcal L(\overline X(T)))(y^{\varepsilon}(T)+z^{\varepsilon}(T))\\
        &\hspace{1cm}+ \widetilde{\mathbb E} \big[h_{\mu}(\overline X(T),\mathcal L(\overline X(T)))(\widetilde{\overline{X}}(T))(\tilde y^{\varepsilon}(T)+\tilde z^{\varepsilon}(T))\big] \Big]\\
       &  +  \mathbb E\Big[\int_0^T \big (\delta f(t)\chi_{E_{\varepsilon}}(t)  
       + \frac{1}{2} f_{xx}(t)[y^{\varepsilon}(t), y^\varepsilon (t)]  
       + \frac{1}{2}\widetilde{\mathbb E}\big[f_{y\mu}(t)(\widetilde{\overline{X}}(t)) [\tilde y^{\varepsilon}(t), \tilde{y}^\varepsilon (t)] \big] \big) dt  \Big]\\
       &  + \mathbb E\Big[\int_0^T \big(\widetilde{\mathbb E}\big[f_{\mu x}(t)(\widetilde{\overline{X}}(t))[\tilde y^{\varepsilon}(t), y^{\varepsilon}(t)]\big] 
       + \frac{1}{2}\widehat{\mathbb E}\widetilde{\mathbb E}\big[f_{\mu\mu}(t)(\widetilde{\overline{X}}(t), \widehat{\overline{X}} (t)) [\tilde y^{\varepsilon}(t), \widehat{y}^{\varepsilon}(t)]\big]\big) dt  \Big]\\
       &  +\frac{1}{2}\mathbb E \Big[ h_{xx}(T)[y^{\varepsilon}(T),y^{\varepsilon}(T)] 
       + \widetilde {\mathbb E}\big[ h_{y\mu}(T)(\widetilde{\overline{X}}(T))[\tilde{y}^{\varepsilon}(T), \tilde{y}^\varepsilon (T) ]\big]\Big]\\
       &  + \mathbb E\Big[\frac{1}{2}\widehat{\mathbb E}\widetilde{\mathbb E}\big[h_{\mu\mu}(T)(\widetilde{\overline{X}}(T), \widehat{\overline{X}}(T))[\tilde y^{\varepsilon}(T), \widehat{y}^{\varepsilon}(T)]\big] 
       + \widetilde {\mathbb E}\big[ h_{\mu x}(T)(\widetilde{\overline{X}}(T))[\tilde y^{\varepsilon}(T), y^\varepsilon (T)]\big]\Big] \\
       &  + o(\varepsilon).   
\end{aligned}
\end{equation}
Here, $(\widehat{\overline{X}}, \widehat{y}^\varepsilon)$ are independent copies of 
$(X, y^\varepsilon )$ (and independent of $(\widetilde{\overline{X}}, \tilde{y}^\varepsilon )$), 
on a separate probability space with corresponding expectation $\widehat{\mathbb{E}}$. 

To get rid of the variational processes, we need the following estimates and duality relations between 
the variational systems \eqref{eq-4.4}-\eqref{eq-4.5} and the adjoint equations \eqref{1-adeq} and 
\eqref{ad-eq2}. By Corollary \ref{cor3.1}, we have 
\begin{equation}\label{eq-4-2}
\begin{aligned}
    &  \Big|\mathbb E\Big[\int_0^T \widetilde{\mathbb E}\big[f_{\mu x}(t)(\widetilde{\overline{X}}(t))[\tilde y^{\varepsilon}(t), y^{\varepsilon}(t)]\big] dt \Big] \Big| 
    = \Big|\int_0^T \mathbb E\Big[\widetilde{\mathbb E}\big[f_{\mu x}(t)(\widetilde{\overline{X}}(t)) [\tilde y^{\varepsilon}(t), y^{\varepsilon}(t)]\big] \Big] dt\Big|\\
    & \leq  \cC\Big( \int_0^T \mathbb E \Big[ \big|\widetilde{\mathbb E}\big[f_{\mu x}(t)(\widetilde{\overline{X}}(t))[\cdot , \tilde y^{\varepsilon}(t)]\big]\big|^2 \Big] dt\Big)^{\frac{1}{2}} 
    \Big(\int_0^T \mathbb E \big[ | y^{\varepsilon}(t)|^2 \big] dt\Big)^{\frac{1}{2}}\\
    & =  o(\varepsilon^{\frac{1}{2}}) \cdot O(\varepsilon^\frac{1}{2})=o(\varepsilon)
\end{aligned}
\end{equation}
and 
\begin{equation}\label{eq-4-3}
\begin{aligned}
   &  \Big|\mathbb E\Big[\int_0^T \widehat{\mathbb E}\widetilde{\mathbb E}\big[f_{\mu\mu}(t)(\widetilde{\overline{X}}(t), \widehat{\overline{X}} (t)) [\tilde y^{\varepsilon}(t), \widehat{y}^{\varepsilon}(t)]\big] dt  \Big]\Big|\\
   &\leq  \mathbb E\Big[\int_0^T \widehat{\mathbb E}\big|\widetilde{\mathbb E} 
   \big[f_{\mu\mu}(t)(\widetilde{\overline{X}}(t), \widehat{\overline{X}} (t)) [\cdot , \tilde y^{\varepsilon}(t)]\big] \big|^2 \Big]^{\frac{1}{2}} \widehat{\mathbb E}\big[\big|\widehat y^{\varepsilon}(t)\big|^2 \big]^{\frac 12} dt \Big]\\
   & =  o(\varepsilon).
\end{aligned}
\end{equation}
Similarly, we have 
\begin{equation}\label{eq-4-4}
    \Big|\widetilde {\mathbb E}\big[ h_{\mu x}(T)(\widetilde{\overline{X}}(T))[\tilde y^{\varepsilon}(T), y^{\varepsilon}(T)]\big]\Big|=  o(\varepsilon).
\end{equation}
and 
\begin{equation}\label{eq-4-5}
    \Big|\mathbb E\widehat{\mathbb E}\widetilde{\mathbb E}\big[h_{\mu\mu}(T)(\widetilde{\overline{X}}(T), \widehat{\overline{X}} (T))[\tilde y^{\varepsilon}(T), \widehat{y}^{\varepsilon}(T)]\big]  \Big|=o(\varepsilon).
\end{equation}

Since $\overline{X}, y^\varepsilon, z^\varepsilon$ and $\widetilde{\overline{X}}, \tilde{y}^\varepsilon, \tilde{z}^\varepsilon$ have identical distribution we can rewrite 
$$
\mathbb{E} \tilde{\mathbb{E}} [f_\mu (t) (\widetilde{\overline{X}}(t)) (\tilde{y}^\varepsilon (t) + \tilde{z}^\varepsilon (t))] 
= \mathbb{E} \tilde{\mathbb{E}} [\tilde{f}_\mu (t) (\overline X (t)) (y^\varepsilon (t) + z^\varepsilon (t))] 
$$
and similar for the term $h_\mu$ in the expansion \eqref{eq-3.11}.  
Therefore we can reduce \eqref{eq-3.11} to  
\begin{equation}\label{eq-3.11-b}
\begin{aligned}
    0 & \leq   \cJ(u^{\e}(\cdot))-\cJ(\bar u(\cdot))\\
    & = \mathbb E\Big[\int_0^T \big\{(f_x(t)+\widetilde{\mathbb E}\big [ \tilde{f}_{\mu}(t)(\overline X(t))\big] )(y^{\varepsilon}(t)+z^{\varepsilon}(t))\big\}dt\Big] \\
        &  + \mathbb E \Big[ (h_x(\overline X(T),\mathcal L(\overline X(T))) 
        + \widetilde{\mathbb E} \big[\tilde{h}_{\mu} (\widetilde{\overline{X}}(T),\mathcal L(\widetilde{\overline{X}}(T)))( \overline X(T))
         (y^{\varepsilon}(T)+z^{\varepsilon}(T))\big] \Big]\\
       &  +  \mathbb E\Big[\int_0^T \big (\delta f(t)\chi_{E_{\varepsilon}}(t) 
       + \frac{1}{2} \big( f_{xx}(t) [y^\varepsilon (t), y^\varepsilon (t)] \\
       & \hspace{1.5cm}+ 
       \widetilde{\mathbb E}\big[  f_{y\mu}(t)(\widetilde{\overline{X}}(t)) [\tilde{y}^{\varepsilon}(t), \tilde{y}^\varepsilon (t)]\big] \big) dt  \Big]\\
       &  +\frac{1}{2}\mathbb E \Big[ h_{xx}(T)[y^{\varepsilon}(T), y^\varepsilon (T)]
       + \widetilde {\mathbb E}\big[  h_{y\mu}(T)(\widetilde{\overline{X}}(T)) [\tilde{y}^{\varepsilon}(T), \tilde{y}^\varepsilon (T)] \big]\Big] +o(\varepsilon).
\end{aligned}
\end{equation}

Next, by the definition of the transposition solution $p$ of \eqref{1-adeq}  
\begin{eqnarray}\label{eq-4.15}
\begin{aligned}
    &\mathbb E\big[\langle p(T),y^{\varepsilon}(T)\rangle\big] =  \mathbb E\Big[\int_0^T \langle y^{\varepsilon}(t), f_x(t)+\widetilde{\mathbb E} \big[\tilde{f}_{\mu}(t)(\overline X(t))\big]\rangle  dt   \Big]\nonumber\\
    &  + \mathbb E\Big[\int_0^T \big(\langle \delta a(t),p(t)\rangle +\langle \delta b(t),q(t)\rangle_{\mathcal L_2^0} \big) \chi_{E_{\varepsilon}}(t)dt\Big]\\
    &  + \mathbb E\Big[\int_0^T\big( \langle p(t),\widetilde{\mathbb E}(\delta a_{\mu}(t)(\widetilde{\overline{X}}(t))\tilde{y}^{\varepsilon}(t))\rangle + \langle q(t),\widetilde{\mathbb E}(\delta b_{\mu}(t)(\widetilde{\overline{X}}(t))\tilde{y}^{\varepsilon}(t))\rangle_{\mathcal L_2^0} \big)dt\Big],
\end{aligned}
\end{eqnarray}
and 
\begin{equation}\label{eq-4.16}
\begin{aligned}
   & \mathbb E\big[\langle p(T),z^{\varepsilon}(T)\rangle\big]  
    =  \mathbb E\Big[\int_0^T \langle z^{\varepsilon}(t), f_x(t)+\widetilde{\mathbb E} \big[\tilde{f}_{\mu}(t)(\overline X(t))\big]\rangle  dt \Big]\\
    & \quad+ \mathbb E\Big[\int_0^T \big(\langle \delta a_x(t)y^{\varepsilon}(t),p(t)\rangle +\langle \delta b_x(t)y^{\varepsilon}(t),q(t)\rangle_{\mathcal L_2^0} \big)\chi_{E_{\varepsilon}}(t)dt\Big]\\
    & \quad + \frac 12 \mathbb E\Big[\int_0^T \langle p(t), a_{xx}(t) [y^{\varepsilon}(t), y^\varepsilon (t) ] 
    + \widetilde{\mathbb E}\big[a_{y\mu}(t)(\widetilde{\overline{X}}(t))[\tilde{y}^{\varepsilon}(t), \tilde{y}^\varepsilon (t)]\big]\rangle dt\Big]\\
    &\quad+ \widetilde{\mathbb E}\big[b_{y\mu}(t)(\widetilde{\overline{X}}(t))[\tilde{y}^{\varepsilon}(t), \tilde{y}^\varepsilon (t)]\big]\rangle_{\mathcal L_2^0} dt\Big]. 
\end{aligned}
\end{equation}

Inserting \eqref{eq-4.15} and \eqref{eq-4.16} into \eqref{eq-3.11-b} and inserting also the terminal condition\\
$p(T)=-h_x(\overline X(T),\mathcal L(\overline X(T)))- \widetilde{\mathbb E} \big[ h_{\mu}(\widetilde{\overline{X}} (T),\mathcal L(\widetilde{\overline{X}}(T)))(\overline{X}(T))\big]$ yields 
\begin{equation}\label{eq-3.11-c}
\begin{aligned}
    0 & \leq   \cJ(u^{\e}(\cdot))-\cJ(\bar u(\cdot))\\
    & = - \mathbb{E}\Big[\int_0^T \frac 12\mathbb{H}_{xx} (t) [y^\varepsilon (t), y^\varepsilon (t)] 
    + \frac 12 \widetilde{\mathbb{E}} \big[ \mathbb{H}_{y\mu} (t)(\widetilde{\overline{X}}(t))[\tilde{y}^\varepsilon (t), \tilde{y}^\varepsilon (t)] \big] + \delta \mathbb{H} (t)\chi_{E_{\varepsilon}}(t)  dt  \big]  \\ 
    &  - \mathbb{E} \Big[ \int_0^T \big( \langle \delta a_x(t) y^\varepsilon (t) , p (t)\rangle 
    + \langle \delta b_x(t) y^\varepsilon (t) , q (t)\rangle_{\mathcal{L}_2^0}\big)\chi_{E_{\varepsilon}}(t)  dt \Big] \\ 
    &   - \mathbb{E} \Big[ \int_0^T \big( \langle \widetilde{\mathbb{E}} \big[ \delta a_\mu (t)(\widetilde{\overline{X}}(t)) \tilde{y}^\varepsilon (t) \big] , p (t)\rangle 
    + \langle \widetilde{\mathbb{E}} \big[ \delta b_\mu (t)(\widetilde{\overline{X}}(t)) \tilde{y}^\varepsilon (t) \big] , q (t)\rangle_{\mathcal{L}_2^0} \big) \chi_{E_{\varepsilon}}(t) dt \Big]  \\ 
    &   +\frac{1}{2}\mathbb E \Big[ h_{xx}(T)[y^{\varepsilon}(T), y^\varepsilon (T)]
       + \widetilde {\mathbb E}\big[  h_{y\mu}(T)(\widetilde{\overline{X}}(T)) [\tilde{y}^{\varepsilon}(T), \tilde{y}^\varepsilon (T)] \big]\Big] +o(\varepsilon).
\end{aligned}
\end{equation}
Furthermore, by the definition of the relaxed transposition solution $P$ of \eqref{ad-eq2} applied to $\varphi_1 = \varphi_2 = y^\varepsilon$, $\xi_1 = \xi_2 = 0$, $u_1 = u_2 = 0$ and $v_1 = v_2 = \delta b \chi_{E_\varepsilon}$, 
\begin{equation}\label{eq-4.17}
\begin{aligned}
    &  \mathbb E \big[\langle P(T)y^{\varepsilon}(T),y^{\varepsilon}(T)\rangle \big]\\
    & =  - \mathbb{E} \Big[\int_0^T \mathbb{H}_{xx}(t) [y^{\varepsilon}(t),y^{\varepsilon}(t)]
     + \widetilde{\mathbb E} \big[ \widetilde{\mathbb H}_{y\mu }(t)(\overline X(t))[y^{\varepsilon}(t),y^{\varepsilon}(t)]\big] dt \Big] \\
  &  \quad  + \mathbb{E} \Big[ \int_0^T \big( \langle P(t) b_{x}(t)  y^{\varepsilon}(t),\delta b(t)\rangle_{\mathcal{L}_2^0}\big)\chi_{E_{\varepsilon}}(t) dt \Big]  \\
  &  \quad  +\mathbb{E} \Big[ \int_0^T \big( \langle P(t)\delta b(t),  b_{x}(t)  y^{\varepsilon}(t)  
  + \delta b(t)\rangle_{\mathcal{L}_2^0}\big) \chi_{E_{\varepsilon}}(t)dt \Big]  \\
  &  \quad  +\mathbb{E}\Big[ \int_0^T \big(\langle \delta b(t),\hat{Q}^{(t)}(0,0,\delta b(t)\chi_{E_{\varepsilon}})(t)\rangle_{\mathcal{L}_2^0}\\
  &\hspace{2cm}+ \langle Q^{(t)}(0,0,\delta b(t)\chi_{E_{\varepsilon}}(t))(t),
 \delta b(t)\rangle_{\mathcal{L}_2^0}\big) \chi_{E_{\varepsilon}}(t)dt \Big]. 
\end{aligned}
\end{equation}
Inserting \eqref{eq-4.17} into \eqref{eq-3.11-c}, using 
$$
\mathbb{E} \widetilde{\mathbb E} \Big[ \widetilde{\mathbb H}_{y\mu}(t)(\overline X(t))
[y^{\varepsilon}(t),y^{\varepsilon}(t)]\Big]
= \mathbb{E}\widetilde{\mathbb E} \Big[\mathbb H_{y\mu}(t)(\widetilde{\overline{X}}(t)) 
[\tilde{y}^{\varepsilon}(t), \tilde{y}^{\varepsilon}(t)]\Big] , 
$$
similarly also for the term $h_{y\mu}$, and using the terminal condition 
$P(T)=- h_{xx}(T)+\widetilde{\mathbb{E}}\big[\tilde{h}_{y\mu}(T)(\widetilde{\overline{X}}(T))\big]$,  
we have 

\begin{equation}\label{eq-4.23}
\begin{aligned}
     0 & \leq   \cJ(u^{\e}(\cdot))-\cJ(\bar u(\cdot))\\
       & =   -\mathbb{E}\Big[\int_0^T  \big( \delta \mathbb{H} (t) + \langle \delta a_x(t) y^\varepsilon (t) , p (t)\rangle + \langle \delta b_x(t) y^\varepsilon (t) , q (t)\rangle_{\mathcal{L}_2^0} \big)\chi_{E_{\varepsilon}}(t) dt \Big] \\ 
    & \quad  - \mathbb{E} \Big[ \int_0^T  \big( \langle \widetilde{\mathbb{E}} \big[ \delta a_\mu (t)(\widetilde{\overline{X}}(t)) \tilde{y}^\varepsilon (t) \big] , p (t)\rangle 
    + \langle \widetilde{\mathbb{E}} \big[ \delta b_\mu (t)(\widetilde{\overline{X}}(t)) \tilde{y}^\varepsilon (t) \big] , q (t)\rangle_{\mathcal{L}_2^0}\big) \chi_{E_{\varepsilon}}(t)  dt \Big] \\ 
  & \quad  - \frac 12 \mathbb{E} \Big[ \int_0^T \big(\langle P(t) b_{x}(t)  y^{\varepsilon}(t),\delta b(t)\rangle_{\mathcal{L}_2^0}\big)\chi_{E_{\varepsilon}}(t) dt \Big]  \\
  &  \quad - \frac 12 \mathbb{E} \Big[ \int_0^T \big( \langle P(t)\delta b(t),  b_{x}(t)  y^{\varepsilon}(t)  
  + \delta b(t)\rangle_{\mathcal{L}_2^0}\big) \chi_{E_{\varepsilon}}(t)dt \Big] \\
  &  \quad  - \frac 12 \mathbb{E}\Big[ \int_0^T \big(\langle \delta b(t),\hat{Q}^{(t)}(0,0,\delta b(t)\chi_{E_{\varepsilon}})(t)\rangle_{\mathcal{L}_2^0}\\
 &\hspace{2.5cm}+ \langle Q^{(t)}(0,0,\delta b(t)\chi_{E_{\varepsilon}}(t))(t),
 \delta b(t)\rangle_{\mathcal{L}_2^0}\big) \chi_{E_{\varepsilon}}(t)dt \Big]\Big\}  +o(\varepsilon).
\end{aligned}
\end{equation}
Now let's get rid of the terms involving $y^{\varepsilon}(t)$. By the boundedness of $a_x$ and $b_x$, together with the definition of $(p(\cdot),q(\cdot))$ and \eqref{eq-4.7}, we see that as $\varepsilon\to 0$, 
\begin{equation}\label{eq-4.24}
\begin{aligned}
    &  \Big|\mathbb E\Big[\int_0^T \big(\langle \delta a_x(t)y^{\varepsilon}(t),p(t)\rangle +\langle \delta b_x(t)y^{\varepsilon}(t),q(t)\rangle_{\mathcal{L}_2^0} \big) \chi_{E_{\varepsilon}}(t)dt \Big] \Big|\\
    & \leq  \cC \mathbb E\Big[ |E_{\varepsilon}|^{\frac{1}{2}}\Big(\int_{E_{\varepsilon}(t)}|p(t)|^2dt\Big)^{\frac{1}{2}}\sup\limits_{t\in[0,T)}|y^{\varepsilon}(t)|+|E_{\varepsilon}|^{\frac{1}{2}}\Big(\int_{E_{\varepsilon}(t)}|q(t)|^2dt\Big)^{\frac{1}{2}}\sup\limits_{t\in[0,T)}|y^{\varepsilon}(t)|\Big]\\
     & \leq  \cC\varepsilon^{\frac{1}{2}} \Big(\mathbb E \Big[\int_{E_{\varepsilon}(t)}|p(t)|^2dt \Big]^{\frac{1}{2}} 
     + \mathbb E\Big[ \int_{E_{\varepsilon}(t)}|q(t)|^2dt \Big]^{\frac{1}{2}}\Big) 
     \mathbb E\big[\sup\limits_{t\in[0,T)}|y^{\varepsilon}(t)|^2\big]^{\frac{1}{2}}\\
     & \leq  \cC\varepsilon \Big( \mathbb E\Big[\int_{E_{\varepsilon}(t)}|p(t)|^2dt\Big]^{\frac{1}{2}} 
     + \mathbb E \Big[\int_{E_{\varepsilon}(t)}|q(t)|^2dt\Big]^{\frac{1}{2}}\Big)=o(\varepsilon).
\end{aligned}
\end{equation}
Similarly, 
{\small
\begin{equation}\label{eq-4.25}
\begin{aligned}
    &  \Big|\mathbb{E} \Big[ \int_0^T \big( \langle P(t) b_x(t)y^{\varepsilon}(t), \delta b(t)\rangle_{\mathcal{L}_2^0} +  \langle P(t)\delta b(t),b_{x}(t) y^{\varepsilon}(t) \rangle_{\mathcal{L}_2^0}\big) \chi_{E_{\varepsilon}}(t)dt \Big] \Big|\\
    & \leq  \cC\mathbb E\Big[|E_{\varepsilon}|^{\frac{1}{2}}\Big(\int_{E_{\varepsilon}(t)}|P(t)|^2dt\Big)^{\frac{1}{2}} \sup\limits_{t\in[0,T)}|y^{\varepsilon}(t)|\Big]\\
     & \leq  \cC\varepsilon^{\frac{1}{2}} \mathbb E \Big[ \int_{E_{\varepsilon}(t)}|P(t)|^2dt\Big]^{\frac{1}{2}}\mathbb E \Big[ \sup\limits_{t\in[0,T)}|y^{\varepsilon}(t)|^2\Big]^{\frac{1}{2}}\\
     & \leq  \cC\varepsilon \mathbb E\Big[ \int_{E_{\varepsilon}(t)}|P(t)|^2dt\Big]^{\frac{1}{2}} =o(\varepsilon).
\end{aligned}
\end{equation}
}
Next, by the boundedness of $a_{\mu}$ and $b_{\mu}$, we get from Corollary \ref{cor3.1} 
\begin{equation}\label{eq-4.26}
\begin{aligned}
  &  \Big| \mathbb E\Big[\int_0^T\big( \langle p(t),\widetilde{\mathbb E}\big[\delta a_{\mu}(t)(\widetilde{\overline{X}}(t))\tilde{y}^{\varepsilon}(t)\big]\rangle + \langle q(t),\widetilde{\mathbb E}\big[\delta b_{\mu}(t)(\widetilde{\overline{X}}(t))\tilde{y}^{\varepsilon}(t)\big]\rangle\big) \chi_{E_\varepsilon} (t) dt \Big]\Big|  \\
    & \leq  \cC\mathbb E\Big[|E_{\varepsilon}|^{\frac{1}{2}}\Big(\int_{E_{\varepsilon}(t)}|p(t)|^2dt\Big)^{\frac{1}{2}}\Big(\int_{E_{\varepsilon}(t)}|\widetilde{\mathbb E}\big[\delta a_{\mu}(t)(\widetilde{\overline{X}}(t))\tilde{y}^{\varepsilon}(t)\big]|^2dt\Big)^{\frac{1}{2}}\Big]\\
    & +\cC\mathbb E\Big[|E_{\varepsilon}|^{\frac{1}{2}}\Big(\int_{E_{\varepsilon}(t)}|q(t)|^2dt\Big)^{\frac{1}{2}}\Big(\int_{E_{\varepsilon}(t)}|\widetilde{\mathbb E}\big[\delta b_{\mu}(t)(\widetilde{\overline{X}}(t))\tilde{y}^{\varepsilon}(t)\big]|^2dt\Big)^{\frac{1}{2}}\Big)\\
      & \leq  \cC\varepsilon^{\frac{1}{2}}\mathbb E\Big[ \int_{E_{\varepsilon}(t)}|p(t)|^2dt\Big)]^{\frac{1}{2}}\Big(\int_0^T \mathbb{E} \big[ |\widetilde{\mathbb E}\big[\delta a_{\mu}(t)(\widetilde{\overline{X}}(t))\tilde{y}^{\varepsilon}(t)\big]|^2 \big] dt \Big)^{\frac{1}{2}}\\
    &   +\cC\varepsilon^{\frac{1}{2}} \mathbb E\Big[\int_{E_{\varepsilon}(t)}|q(t)|^2dt\Big]^{\frac{1}{2}}\Big(\int_0^T \mathbb E\big[ |\widetilde{\mathbb E}\big[\delta b_{\mu}(t)(\widetilde{\overline{X}}(t))\tilde{y}^{\varepsilon}(t)\big]|^2 \big] dt\Big)^{\frac{1}{2}}\\
     & \leq  \cC\varepsilon \Big\{\Big(\mathbb E \Big[ \int_{E_{\varepsilon}(t)}|p(t)|^2dt\Big]^{\frac{1}{2}} + \mathbb E\Big[ \int_{E_{\varepsilon}(t)}|q(t)|^2dt\Big]^{\frac{1}{2}} \Big\}\\
     & = o(\varepsilon).
\end{aligned}
\end{equation}
Note that the terms containing $\hat{Q}^{(t)}(0,0,\delta b(t)\chi_{E_{\varepsilon}})(t)$ and $Q^{(t)}(0,0,\delta b(t)\chi_{E_{\varepsilon}})(t)$ admit the same properties as in the standard non-McKean-Vlasov SPDE case (discussed in \cite[Chapter 9]{Lu2014}), therefore,
\begin{eqnarray}\label{eq-4-14}
  & &  \Big|\mathbb{E}\Big[ \int_0^T \langle \delta b(t),\hat{Q}^{(t)}(0,0,\delta b(t)\chi_{E_{\varepsilon}})(t)\rangle_{\mathcal{L}_2^0} \chi_{E_{\varepsilon}}(t) dt \Big]\Big|\nonumber\\
   & & \leq \big|\hat{Q}^{(0)}(0,0,\delta b(t)\chi_{E_{\varepsilon}}(\cdot))\big|_{L_\mathbb{F}^2(0,T;L^{4/3}(\Omega;\mathcal{L}_2^0))}\mathbb{E}\Big[ \Big(\int_0^T |\delta b(t)\chi_{E_{\varepsilon}}(t)|_{\mathcal{L}_2^0}^4dt\Big)^{1/2}\Big]^{1/2}\nonumber\\
  & & \leq C|\delta b(\cdot)\chi_{E_{\varepsilon}}(\cdot) |_{L_\mathbb{F}^2(0,T;L^4(\Omega,\mathcal{L}_2^0))} \mathbb{E}\Big[\Big( \int_0^T |\delta b(t)\chi_{E_{\varepsilon}}(t)|_{\mathcal{L}_2^0}^4 dt \Big)^{1/2} \Big]^{1/2}\\
  & & =o(\varepsilon), \nonumber
\end{eqnarray}
and 
\begin{eqnarray}\label{eq-4-15}
\Big|\mathbb{E}\Big[\int_0^T\langle Q^{(t)}(0,0,\delta b(t)\chi_{E_{\varepsilon}}(t))(t),
 \delta b(t)\rangle_{\mathcal{L}_2^0}\chi_{E_{\varepsilon}}(t)dt\Big]\Big| 
 = o(\varepsilon)
\end{eqnarray}

Combining \eqref{eq-4.23}-\eqref{eq-4-15}, we get
\begin{eqnarray}\label{eq-4.28}
    0 & \leq  & \cJ(u^{\e}(\cdot))-\cJ(\bar u(\cdot))\nonumber\\
    & =& -\mathbb E\Big[\int_0^T \delta \mathbb H(t)\chi_{E_{\varepsilon}}(t)+ \frac 12\langle P(t)\delta b(t),\delta b(t)\rangle \chi_{E_{\varepsilon}}(t)dt\Big] + o(\varepsilon).
\end{eqnarray}
Finally, from Lebesgue differentiation theorem, we deduce from \eqref{eq-4.28}, for any $u\in U$, and
$a.e.\  t\in [0,T]$, it holds $\mathbb P$-almost surely that 
\begin{eqnarray*}
    0& \leq& \mathbb H(t,\overline X(t), \mathcal L (\overline X(t)), \bar u(t), p(t),P(t))- \mathbb H(t,\overline X(t), \mathcal L (\overline X(t)), u, p(t),P(t))\\
    & & -\frac{1}{2} \big\langle P(t)\big (b(t)-b(t,\overline X(t),\mathcal L(\overline X(t)),u)\big ),b(t)-b(t,\overline X(t),\mathcal L(\overline X(t)),u)\big\rangle_{\mathcal L_2^0}.
\end{eqnarray*}
This proves Theorem \ref{thm4.1}.

\section{Example}
In this section, we provide an example to illustrate the application of our McKean-Vlasov SPDE control problem. 

Let $G\subset\dbR^n$ be a bounded domain with the smooth boundary
$\pa G$.  Let $H=L^2(G)$ and $U$ be a bounded closed subset of
$L^2(G)$. Consider the following stochastic parabolic equation:
\begin{equation}\label{system3-a}
\begin{cases}
\ds   dy =\big(\Delta y + \tilde a(t,y,\mathcal L(y),u)\big) dt+\tilde b(t,y,\mathcal L(y),u) dW(t) &\textup{in }  (0,T]\times G,\\
\ns\ds   y =0 &\textup{on }   (0,T]\times \pa G,\\
\ns\ds   y(0)=\xi(x) &\textup{in }  G,
\end{cases}
\end{equation}
where $ x \in L^2(G)$, $u(\cdot)\in \cU[0,T]$, $\xi\in L^2_{\mathcal F_0}(\Omega;L^2(G))$,

and the following
cost functional:
\begin{equation}\label{cost3}
\mathcal{J}(x;u(\cdot))= \mathbb{E}\Big(\int_0^T\int_G \tilde
f(t,y(t),\mathcal L(y(t)),u(t))dxdt+\int_G \tilde h(y(T),\mathcal L(y(T)))dx \Big),
\end{equation}
with 
$\tilde a, \tilde b$ and
$\tilde f,\tilde h$  satisfy the following condition:

\ss

\no{\bf (C)} {\it For $\f=\tilde a,\tilde b, \tilde f, \tilde h$, and for all $(t,u)\in [0,T]\times U$,

(i) $\tilde a(t,\cdot,\cdot,u)\in C_b^{2,2}(H\times \mathcal P_2(H); H)$,

(ii) $\tilde b(t,\cdot,\cdot,u)\in C_b^{2,2}(H\times \mathcal P_2(H); \mathcal L_2^0)$,

(iii) $\tilde f(t,\cdot,\cdot,u)\in C_b^{2,2}(H\times \mathcal P_2(H); \mathbb R)$,

(iv) $\tilde h(\cdot,\cdot)\in C_b^{2,2}(H\times \mathcal P_2(H); \mathbb R)$.}

\ss

Under {\bf (C)}, it is easy to see that ({\bf
A}) hold. Then we know that all assumptions in Theorem
\ref{thm4.1} are fulfilled. This formulation can be applied to neuroscience. In particular, \eqref{system3-a}-\eqref{cost3} provide a spatially extended analogue of the finite-dimensional McKean-Vlasov SDE control problem studied in \cite{Hocquet2021}.

\section*{Acknowledgements}

LC and WS acknowledge support from DFG CRC/TRR 388 ’Rough Analysis, Stochastic
Dynamics and Related Fields’, Project A10.


\end{document}